\documentclass[11pt]{article}

\usepackage{amsmath, amsfonts, amssymb, amsthm}
\usepackage{graphicx}
\usepackage{color,ulem}
\usepackage[margin=.94in]{geometry}

\usepackage{caption}
\usepackage{subcaption}

\newcommand{\Rea}{\mathbb{R}}

\newcommand{\Nat}{\mathbb{N}}

\newcommand{\var}{\operatorname{Var}}

\newcommand{\ud}{\mathrm{d}}

\newcommand{\norm}[1]{\| #1 \|}

\newtheorem{theorem}{Theorem}
\newtheorem{lemma}[theorem]{Lemma}

\newtheorem{proposition}[theorem]{Proposition}

\newtheorem{conjecture}[theorem]{Conjecture}

\theoremstyle{definition}
\newtheorem{definition}[theorem]{Definition}

\begin{document}

\title{{\bf Randomized longest-queue-first scheduling \\ for large-scale buffered systems}}
\date{\today}
\author{A. B. Dieker \and T. Suk}

\maketitle

\begin{abstract}
We develop diffusion approximations for parallel-queueing systems with the randomized longest-queue-first scheduling algorithm by establishing new mean-field limit theorems as the number of buffers $n\to\infty$.
We achieve this by allowing the number of sampled buffers $d=d(n)$ to depend on the number of buffers $n$, which yields an asymptotic `decoupling' of the queue length processes.

We show through simulation experiments that the resulting approximation is accurate even for moderate values of $n$ and $d(n)$.   To our knowledge, we are the first to derive diffusion approximations for a queueing system in the large-buffer mean-field regime.
Another noteworthy feature of our scaling idea is that the randomized longest-queue-first algorithm emulates the longest-queue-first algorithm, yet is computationally more attractive.
The analysis of the system performance as a function of $d(n)$ is facilitated by the multi-scale nature in our limit theorems: the various processes we study have different space scalings. This allows us to show the trade-off between performance and complexity of the randomized longest-queue-first scheduling algorithm.
\end{abstract}

\section{Introduction}
\label{sec:introduction}

Resource pooling is becoming increasingly common in modern applications of stochastic systems, such as in computer systems, wireless networks,
workforce management, call centers, and health care delivery.
At the same time, these applications give rise to systems which continue to grow in size.
For instance, a traditional web server farm only has a few servers, while cloud data centers have thousands of processors.
These two trends pose significant practical restrictions on admission, routing, and scheduling decision rules or algorithms.
Scalability and computability are becoming ever more important characteristics of decision rules, and consequently
simple decision rules with good performance are of particular interest. An example is the so-called least connection rule 
implemented in many load balancers in computer clouds, which assigns a task to the server with the least number of active connections;
cf.~the join-the-shortest-queue routing policy.
From a design point of view, the search for desirable algorithmic features often presents trade-offs between system performance, information/communication, and required computational effort.

Over the past decades, mean field models have become mainstream aids in the design and performance assessment of 
large-scale stochastic systems, see for instance \cite{Bakhshi:2009ud, Benaim:2008kv,GastN:2010,LeBoudec:2007ux,BVanHoudt:2013}.
These models allow for summary system dynamics to be approximated using a mean-field scaling,
which leads to deterministic `fluid' approximations.
Although these approximations are designed for large systems,
they typically do not work well unless the scaling parameter $n$ is excessively large.

In the view of this, it is of interest to find more refined approximations than fluid approximations. In this paper, we derive diffusion approximations in a specific instance of a large-scale stochastic system:
a queueing system with many buffers with a randomized longest-queue-first scheduling algorithm.
Under this scheduling algorithm, the server works on a task from the buffer with the 
longest queue length among several sampled buffers; it approximates the longest-queue-first scheduling policy,
but it is computationally more attractive if the number of buffers is large.

\paragraph{Our model.}
In our model, each buffer is fed with an independent stream of tasks, which arrive according to a Poisson process.
All $n$ buffers are connected to a single centralized server.
Under the randomized longest-queue-first policy, this server selects $d(n)$ 
buffers uniformly at random (with replacement) and processes a task from the 
longest queue among the selected buffers; it idles for a random amount of time if all buffers in the sample are empty. 
Tasks have random processing time requirements.
The total processing capacity scales linearly with $n$ and the processing time distribution is independent of $n$.
We work in an underloaded regime,
with enough processing capacity to eventually serve all arriving tasks.
Note that this scheduling algorithm is agnostic in the sense that it does not use arrival rates. 
By establishing limit theorems, we develop approximations for the queue length processes in the system,
and show that the approximations are accurate even for moderate $n$ and $d(n)$. 
Also, we study the trade-off between performance and complexity of the algorithm.

\paragraph{Related works.}
Most existing work on the mean-field large-buffer asymptotic regime for queueing systems concentrates on the so-called supermarket model, 
which has received much attention over the past decades following the work of Vvedenskaya {\em et al.}~\cite{Vvedenskaya:1996us}; see also \cite{Mitzenmacher:1996vh} and follow-up work.
The focus of this line of work lies on the question how 
incoming tasks should be routed to buffers, i.e., the load balancing problem.
For the randomized join-the-shortest-queue routing policy
where tasks are routed to the buffer with the shortest queue length among $d$ uniformly selected buffers,
this line of work has exposed a dramatic improvement in performance when $d=2$ versus $d=1$.
This phenomenon is known as the {\em power of two choices}.
A recently proposed different approach for the load balancing problem is inspired by 
the cavity method \cite{Bramson:2010vw,Bramson:2011ut,Bramson:2011tk}. 
This approach is a significant advance in the state-of-the-art since it 
does not require exponentially distributed service times.
However, applying this methodology to our setting presents significant challenges due to 
the scaling employed here.
We do not consider this method here, it remains an open problem whether the cavity
method can be applied to our setting.

The papers by Alanyali and Dashouk \cite{Alanyali:2008vf} and Tsitsiklis and Xu \cite{Tsitsiklis:wz}
are closely related to the present paper. Both consider scheduling in the presence of a large number of buffers. 
The paper \cite{Alanyali:2008vf} studies the randomized longest-queue-first policy with $d(n)=d$, 
and the main finding is that the empirical distribution of the queue lengths in the buffer is asymptotically geometric
with parameter depending on $d$. 
It establishes an upper bound on the asymptotic order, but here we establish tightness and identify the limit. A certain time scaling that is not present in \cite{Alanyali:2008vf} is essential for the validity of our limit theorems.
The paper \cite{Tsitsiklis:wz} analyzes
a hybrid system with centralized and distributed processing capacity in a setting similar to ours.
Their work exposes a dramatic improvement in performance in the presence of centralization
compared to a fully distributed system.

\paragraph{Our contributions.}
We establish a diffusion limit theory for a queueing system in the large-buffer mean-field regime. Diffusion approximations are well-known to arise in the context of mean-field models (e.g., \cite{kurtz:1978kq})
but off-the-shelf results typically cannot directly be applied due to intricate dependencies or technical intricacies.
Thus, by and large, second-order diffusion approximations have been uncharted territory for many large-scale queueing systems. 

Our analysis is facilitated by the idea to scale the number of sampled buffers $d(n)$ with the number of buffers $n$, 
which asymptotically `decouples' the buffers and consequently removes certain dependencies among the buffer contents.
The decoupling manifests itself through a limit theorem on multiple scales, where the various queue-length processes we study
have different space scalings. 
We show empirically that this result leads to accurate approximations even when the 
number of buffers $n$ is small, i.e., outside of the asymptotic regime that motivated the approximation.

For our system, since the scheduling algorithm depends on $n$, several standard arguments 
for large-scale systems break down due to the multi-scale nature of the various stochastic processes involved;
thus, our work requires several technical novelties.
Among these is an induction-based argument for establishing the existence of a fluid model.
We also rely on an appropriate time scaling, 
which is specific to our case and has not been employed in other work.

Our fluid limit theory makes explicit the trade-off between performance and complexity for our algorithm.
Intuitively, one expects better system performance for larger $d(n)$, since the likelihood of idling decreases;
however, the computational effort also increases since one must sample (and compare) the queue length of more buffers.
Our main insight into the interplay between performance (i.e., low queue lengths) and computational complexity of the scheduling algorithm within our model can be summarized as follows.
We study the fraction of queues with at least $k$ tasks, and show that it is of order $1/d(n)^k$
under the randomized longest-queue-first scheduling policy. 
This strengthens and generalizes the upper bound from \cite{Alanyali:2008vf}.
Thus, the average queue length is of order $1/d(n)$ as 
$n$ approaches infinity. This should be contrasted with $d(n)$, which is 
the order of the computational complexity of the scheduling algorithm. 

The randomized longest-queue-first algorithm approximates the longest-queue-first algorithm, which is a fully centralized policy, 
so it is appropriate to make a comparison with the partially centralized scheduling algorithm from \cite{Tsitsiklis:wz}, where all $n$ buffers are used with probability $p>0$ (and one buffer is chosen uniformly at random otherwise).
Our algorithm has better performance 
although it compares only $d(n)\ll n$ buffers per job as opposed to $p n+1-p$, which is the average number of buffers used in the partially centralized algorithm.  

\paragraph{Outline of this paper.}
We introduce our model in Section~\ref{sec:model_and_notation}. Our main results come in two pieces: limit theorems (Section~\ref{sec:main_results}) and approximations with validation (Section~\ref{sec:numerics}). Sections~\ref{sec:proofs} contains the proofs of our limit theorems.
Finally, Appendix~\ref{sec:appendix} has several standard results that we have included for quick reference.

\section{Model and notation} 
\label{sec:model_and_notation}
The systems we are interested in consist of many parallel queues and a single server. 
Consider a system with $n$ buffers, which temporarily store tasks to be served by the (central) server.
The number of tasks in a buffer is called its queue length.
Buffers temporarily hold tasks in anticipation of processing, 
and tasks arrive according to independent Poisson processes with rate $\lambda<1$.
The processing times of the tasks are i.i.d.~with an exponential distribution with unit mean.
All processing times are independent of the arrival processes.
The server serves tasks at rate $n$.

The server schedules tasks as follows. It selects $d(n)$ buffers uniformly at random (with replacement) and processes 
a task in the buffer with the longest queue length among the selected buffers.
Ties are broken by selecting a buffer uniformly at random among those with the longest queue length.
If all selected buffers are empty, then the service opportunity is wasted and 
the server waits for an exponentially distributed amount of time with parameter $n$ before resampling.
Once a task has been processed, it immediately leaves the system.
We do not consider scheduling within buffers, since we only study queue lengths.
Throughout, 
we are interested in the case when $d(n)$ satisfies $d(n)=o(n)$ and $\lim_{n\to\infty} d(n)=\infty$.

In this model description, it is not essential that there is exactly one server. Indeed,
the same dynamics arise if an arbitrary number $M$ of servers process tasks at rate $n/M$, 
as long as each server uses the randomized longest-queue-first policy. 
This model arises in the content of cellular data communications \cite{Alanyali:2008vf}.
An abstract representation of the model is displayed in Figure~\ref{fig:diagram}.

\begin{figure}[htbp]
\centering
\begin{subfigure}[b]{0.4\textwidth}
\includegraphics[width=\textwidth]{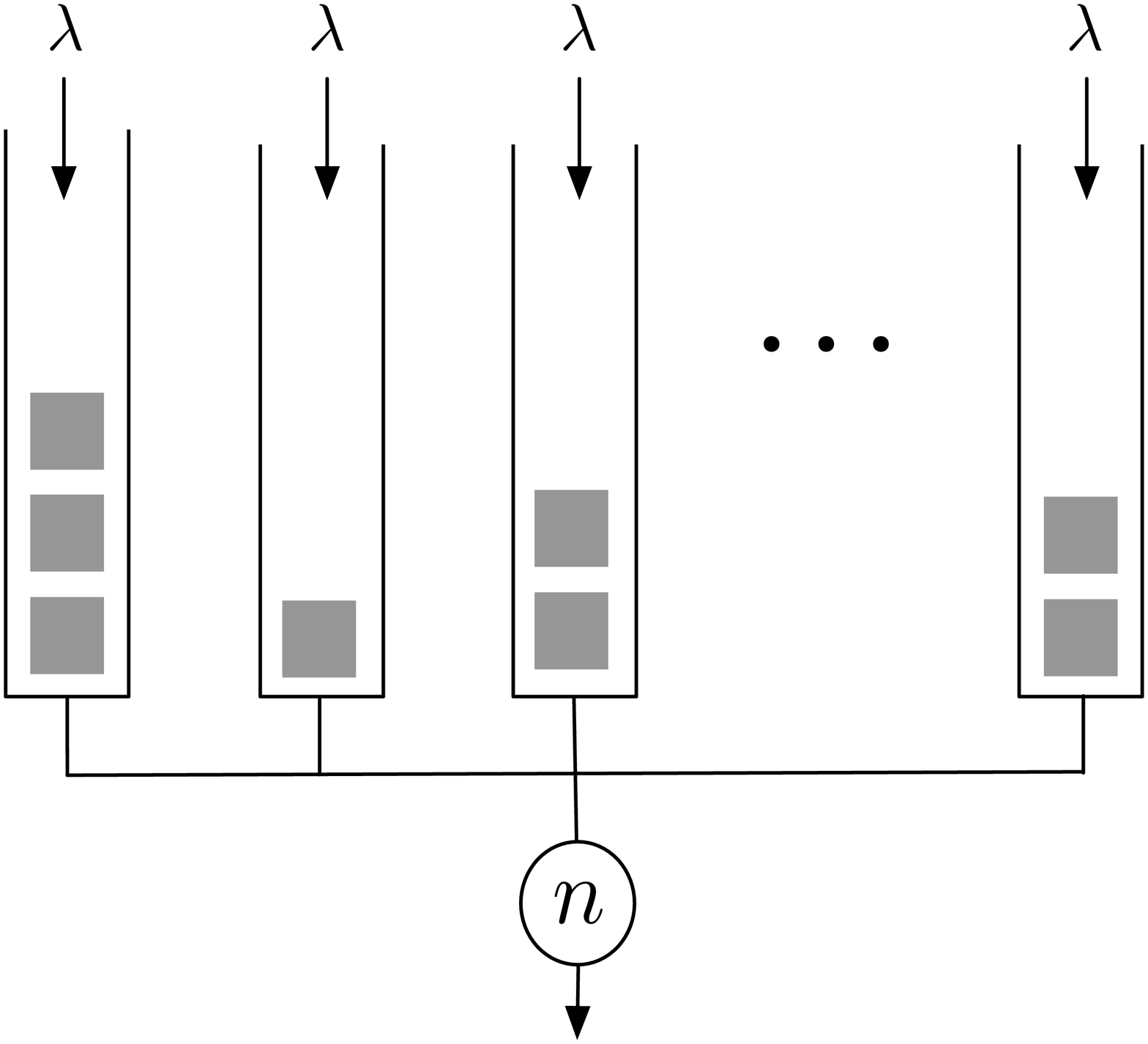}
\end{subfigure}
~
\centering
 \begin{subfigure}[b]{0.4\textwidth}
\includegraphics[width=\textwidth]{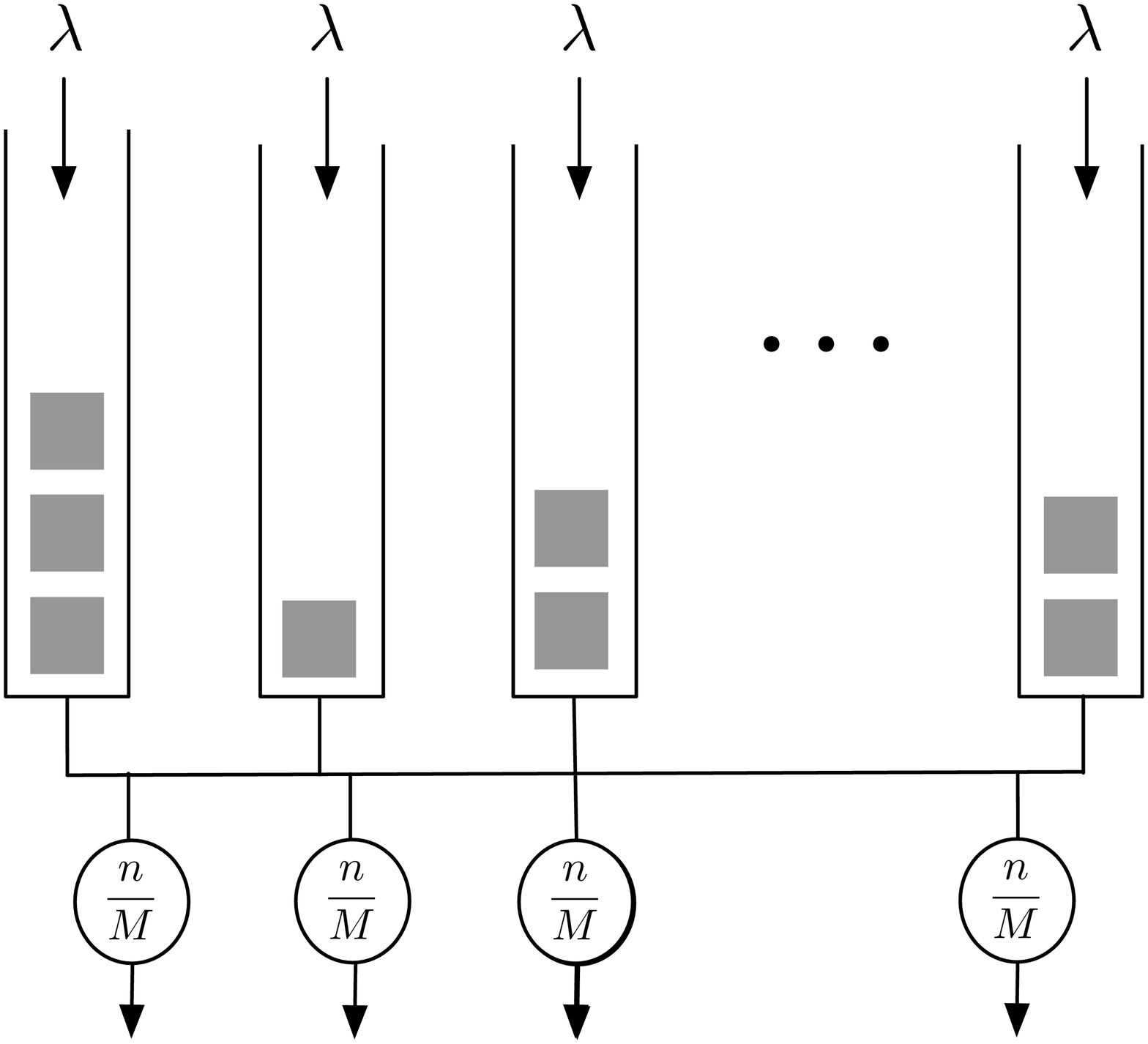}
\end{subfigure}
\caption{Our models with $n$ buffers. Left: One central server with service rate $n$. Right: $M$ servers with service rates $n/M$.}
\label{fig:diagram}
\end{figure}

Let $F_{n,k}(t)$ be the fraction of buffers with queue length greater than or equal to $k$ at time $t$ in the system with $n$ buffers, so that
$\left\{F_{n,k}(t)\right\}_{k\in\mathbb N}$ is a Markov process.
Such mean-field quantities have been used in analyzing various scheduling and load balancing 
policies, e.g., \cite{Alanyali:2008vf,Mitzenmacher:1996vh,Tsitsiklis:wz}. 
However, under the randomized longest-queue-first policy, we can expect
from \cite{Alanyali:2008vf} that, whenever $\lim_{n\to\infty}d(n)=\infty$,
\begin{equation*}
	\lim_{t\to\infty}\lim_{n\to\infty} F_{n,k}(t)=0
\end{equation*}
for all $k\geq 1,$ i.e., in this sense the performance is asymptotically the same as that of the longest-queue-first policy, and these random variables are asymptotically degenerate.

\section{Limit theorems} 
\label{sec:main_results}

In this section, we present limit theorems which are stated in terms of $F_{n,k}(\,\cdot\,)$ under appropriate scaling.
Let $K\in\Nat$ be a fixed finite integer satisfying $\lim_{n\to\infty}n/d(n)^K=\infty$. 
Let $U_{n,k}(\,\cdot\,)$ be the following modification of $F_{n,k}(\,\cdot\,):$
\begin{equation*}
	U_{n,k}(t)~:=~d(n)^k\,F_{n,k}\!\left(\frac{t}{d(n)}\right),
\end{equation*}
for $k=0,1,\dots,K.$
Our first limit theorem is that 
$\left\{ (U_{n,1}(t),\dots,U_{n,K}(t))\right\}_{n\in\mathbb N}$
has a fluid limit as $n\to\infty$ and that this fluid limit satisfies the system of 
differential equations described in the following definition.

\begin{definition}\label{def:fluidSystem}
	For $v_1,\dots,v_K\in\Rea_+,$ $\left(u_1(t),\dots,u_K(t)\right)$ is said to be a \emph{longest-queue-first fluid limit system} with initial condition $(v_1,\dots,v_K)$ if:
	\begin{enumerate}
		\item[(1)] $u_k:[0,\infty)\to\Rea_+$ with $u_k(0)=v_k$~for all~$k=1,\dots,K.$
		\item[(2)] $u_1'(t)=e^{-u_1(t)}-1+\lambda.$
		\item[(3)] $u_{k}'(t)=\lambda\,u_{k-1}(t)-u_k(t),$ for all $k=2,\dots,K.$
	\end{enumerate}
\end{definition}

By the usual existence and the uniqueness theorem of first order ordinary differential equations (e.g., \cite{Braun:1992wv}), 
there is a unique differentiable function $u_1:[0,\infty)\to\Rea_+$ with $u_1(0)=v_1$ satisfying the second condition in Definition \ref{def:fluidSystem}.
For $k\geq 2,$ when $u_{k-1}(t)$ and $v_k$ are given, the differential equation of $u_k$ is linear with inhomogeneous part $u_{k-1}(t)$, and 
therefore $u_k:[0,\infty)\to\Rea_+$ is unique. Thus, by induction, for any given initial condition, there is a unique longest-queue-first 
fluid limit system. 

We remark that the following is an explicit expression of the solution if $v_1< \ln\! \left(\frac{1}{1-\lambda}\right)$ (the other case yields a similar expression):
\begin{eqnarray*}
  && u_1(t)~=~\ln\!\left(  \frac{C_1\,e^{(1-\lambda)t}-1}{C_1(1-\lambda)\,e^{(1-\lambda)t}} \right),\\
	&& u_k(t)~=~e^{-t}\,v_k + \lambda\!\int_0^t\!e^{-(t-s)}\,\,u_{k-1}(s)\,\ud s,\quad\quad k=2,\dots,K,
\end{eqnarray*}
where $C_1={1}/{(1-(1-\lambda)e^{v_1})}$.
Moreover, a longest-queue-first fluid limit system has a unique critical point which is stable: $\left(\ln\!\left( \frac{1}{1-\lambda}\right),\lambda\,\ln\!\left( \frac{1}{1-\lambda}\right),\dots, \lambda^{K-1}\,\ln\!\left( \frac{1}{1-\lambda}\right)\right).$ 
The following proposition summarizes these arguments.

\begin{proposition}\label{prop:fluidLimit}
	For any $(v_1,\dots,v_K)\in\Rea_+^K,$ there is a unique longest-queue-first fluid limit system $\left(u_1(t),\dots,u_K(t)\right)$ with $u_k(0)=v_k$ for all $k=1,\dots,K,$ and 
	\begin{equation*}
		\left(u_1(t),u_2(t),\dots,u_K(t)\right)~\to~\left(\ln\!\left( \frac{1}{1-\lambda}\right),\lambda\,\ln\!\left( \frac{1}{1-\lambda}\right),\dots, \lambda^{K-1}\,\ln\!\left( \frac{1}{1-\lambda}\right)\right)
	\end{equation*}
	as $t\to\infty$.
\end{proposition}

Our first limit theorem states that, with an appropriate initial condition, $\left(U_{n,1}(t),\dots,U_{n,K}(t)\right)$ converges to a fluid limit system as $n\to\infty$.

\begin{theorem}[Fluid limit]\label{thm:mainFluid}
	Consider a sequence of systems indexed by $n$. Fix a number $K\in\Nat$ such that $\lim_{n\to\infty} n/d(n)^K=\infty$. Assume that $U_{n,k}(0)$ is deterministic for every $n$ and $k\le K$, and that 
	there exist $v_1,\dots,v_K\in\Rea_+$ such that
	\begin{equation*}
		\lim_{n\to\infty} U_{n,k}(0)=v_k,\quad\quad k=1,\dots,K,
	\end{equation*}
	and
	\begin{equation*}
		\lim_{n\to\infty} d(n)^K \left(F_{n,K+1}(0)+F_{n,K+2}(0)+\cdots\right)=0.
	\end{equation*}
	Then the sequence of stochastic processes 
  $\left\{ \left(U_{n,1}(t),\dots,U_{n,K}(t)\right)\right\}_{n\in\mathbb N}$ 
  converges 
almost surely to the longest-queue-first fluid limit system $\left(u_1(t),\dots,u_K(t)\right)$ with initial condition $(v_1,\dots,v_K)$, uniformly on compact sets.
\end{theorem}

The proof of the above theorem is based on mathematical induction, and 
we give a high-level overview of this proof at the beginning of Section~\ref{sec:proofs}.

This result makes the explicit trade-off between performance and complexity for randomized longest-queue-first algorithms. 
Theorem~\ref{thm:mainFluid} shows that for $k=1,\dots,K$, as $n\to\infty$,
\begin{equation*}
	F_{n,k}\!\left(\frac{t}{d(n)}\right)=\Theta\!\left(\frac{1}{d(n)^k}\right).
\end{equation*}
For $k=1$, this agrees with the upper bound sketched in \cite{Alanyali:2008vf}. Then the average queue length is order of $\frac{1}{d(n)}$, inverse of the complexity. In the next section, we investigate this by simulation.

\begin{figure}[htbp]
\centering
\includegraphics[width=3in]{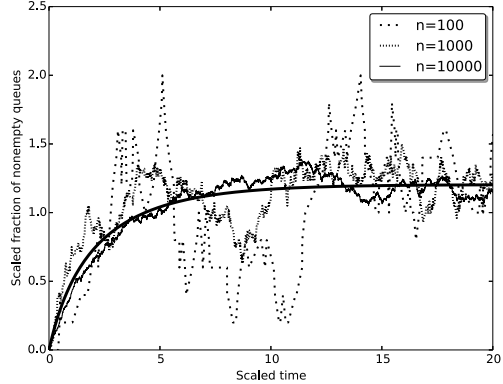}
\caption{Sample paths of $U_{n,1}(t)$ for various $n$, with $d(n)=10\!\cdot\! \log_{10}(n)$ and $\lambda=0.7$. The thick curve is the solution of $u^\prime(t)=e^{-u(t)}-1+\lambda$.}
\label{fig:sample_path}
\end{figure}

Figure~\ref{fig:sample_path} shows sample paths of $U_{n,1}(t)$ (the scaled fraction of nonempty queues) for various $n$ and it empirically confirms our first limit theorem. However, even for $n$ as large as $10000$, the sample paths fluctuate around the fluid limit, especially for large $t$. This means that it is important to incorporate a second-order approximation.

Our second limit theorem is about the diffusion limit of $U_{n,1}(t)$ as $n\to\infty$. Precisely, we show that the stochastic processes $U_{n,1}(t)$ converges in distribution to a diffusion process after appropriate scaling. 
We believe it is the first diffusion limit theorem for a queueing system in the large-buffer mean-field regime, and is based on an asymptotic `decoupling' of the queue length processes.
Note that $U_{n,1}(t)$ is not a Markov process, but the approximating process $Z(t)$ is a Markov process.
In the appendix, we explain the exact meaning of this type of convergence, for which we use the
symbol `$\Rightarrow$'.

\begin{theorem}[Diffusion limit]\label{thm:mainDiffusion}
  Consider a sequence of system indexed by $n$. Suppose that $\lim_{n\to\infty} n/d(n)=\infty$ and $\lim_{n\to\infty} n/d(n)^2=0$. Assume that $U_{n,1}(0)$ is deterministic for all $n$, and that
  there exists some $v_1\in\Rea_+$ such that
  \begin{equation}
    \lim_{n\to\infty} \sqrt{\frac{n}{d(n)}}\left( U_{n,1}(0)-v_1 \right)=0 , \label{eq:diffusion_condition_1}
  \end{equation}
  and
  \begin{equation}
    	\lim_{n\to\infty} \sqrt{n\,d(n)} \left(F_{n,2}(0)+F_{n,3}(0)+\cdots\right)=0.  \label{eq:diffusion_condition_2}
  \end{equation}
  Then we have, as $n\to\infty$,
  \begin{equation*}
    \sqrt{\frac{n}{d(n)}} \left( U_{n,1}(t)-u_1(t) \right) ~\Rightarrow~ Z(t),
  \end{equation*}
  where $Z(t)$ is the solution of the following Ito integral equation:
  \begin{equation*}
    Z(t)=\sqrt{\lambda}\,B^{(1)}(t)-\int_0^t\!\sqrt{1-e^{-u_1(s)}}\,\ud B^{(2)}(s) -\int_0^t e^{-u_1(s)}\,Z(s)\,\ud s 
  \end{equation*}
  for independent Wiener processes $B^{(1)}(t)$ and $B^{(2)}(t)$.
\end{theorem}

We anticipate that this theorem can be generalized as follows. The process $U_{n,k}(t)$ couples with $u_{k+1}(t)$ (the scaling limit of $U_{k+1}(t)$), but the fact that their scaling behavior is different ($\sqrt{n/d(n)^k}$ vs.~$\sqrt{n/d(n)^{k+1}}$) introduces complications for the proof technique used for Theorem~\ref{thm:mainDiffusion}.

\begin{conjecture}\label{con:mainConjecture}
  Consider a sequence of system indexed by $n$. Suppose that $\lim_{n\to\infty} n/d(n)=\infty$ and fix $k\le K$, where $K$ is defined in the beginning of this section. Assume that $U_{n,k}(0)$ is deterministic for all $n$ and $k\leq K$, and that
  there exists $v_1,\dots,v_K\in\Rea_+$ and $v_1^*,\dots,v_K^*\in\Rea$ such that
  \begin{equation*}
    \lim_{n\to\infty} \sqrt{\frac{n}{d(n)^k}}\left( U_{n,k}(0)-v_k \right)=v_k^*.
  \end{equation*}
  Additionally, assume that
  \begin{equation*}
      \lim_{n\to\infty} \sqrt{n\:d(n)^{K+1}} \left(F_{n,K+1}(0)+F_{n,K+2}(0)+\cdots\right)=0.
  \end{equation*}
  Then we have, as $n\to\infty$,
  \begin{equation*}
    \sqrt{\frac{n}{d(n)^k}} \left( U_{n,k}(t)-u_k(t)+ \frac{1}{d(n)} u_{k+1}(t) \right) ~\Rightarrow~ Z_k(t),
  \end{equation*}
  where we interpret $u_{K+1}(t)$ as zero, and $Z_1(t)$ is the solution of the following Ito integral equation:
  \begin{equation*}
    Z_1(t)= v_1^*+\sqrt{\lambda}\,B^{(1)}_1(t)-\int_0^t\!\sqrt{1-e^{-u_1(s)}}\,\ud B^{(2)}_1(s) -\int_0^t e^{-u_1(s)}\,Z_1(s)\,\ud s,
  \end{equation*}
  and, for $k=2,\dots,K$, $Z_k(t)$ is the solution of the following Ito integral equation:
  \begin{equation*}
    Z_k(t)=v_k^*+\int_0^t\!\sqrt{\lambda\,u_{k-1}(s)}\:\ud B^{(1)}_{k}(s)-\int_0^t\!\sqrt{u_k(s)}\,\ud B^{(2)}_{k}(s) -\int_0^t Z_k(s)\,\ud s,
  \end{equation*}
  for independent Wiener processes $B^{(1)}_k(t)$ and $B^{(2)}_k(t)$.
\end{conjecture}

Next, we utilize above our limit theorems to establish approximations of the processes in our system and show their accuracy by simulation.

\section{Approximation and validation}
\label{sec:numerics}
In this section, we propose diffusion approximations based on our limit theorems in the previous section, and 
we investigate the discrepancy between these approximations and the original pre-limit system.
In addition, we examine the trade-off between performance (average queue length) and complexity (the number of samples) through simulation.

Our limit theorems are stated in terms of a function $d(n)$, but here we investigate systems for which we sample a fixed number of buffers $d$. For simplicity, we only consider systems that are initially empty.

\subsection{Diffusion approximations}

Our diffusion limit theorem suggests the following approximation for the distribution of the fraction of nonempty queues in a system with $n$ buffers and $d$ samples:
\begin{equation*}
  \tag{Diffusion Approximation}
  F_{n,1}(t) ~\approx~
  \frac{1}{d} u_1(d t) + \frac{1}{\sqrt{n d}} Z(d t),
\end{equation*}
where $u_1(t)$ is the fluid limit of $U_{n,1}(t)$ from Theorem~\ref{thm:mainFluid} and $Z(t)$ is the Gaussian process defined in Theorem~\ref{thm:mainDiffusion}. One of the assumptions in Theorem~\ref{thm:mainDiffusion} is $\lim_{n\to\infty} n/d(n)^2=0$, which may not be plausible for systems with relatively small $d$ compared to $n$; we confirm this later. Our conjecture in Section~\ref{sec:main_results} suggests adjusting the Diffusion Approximation as follows:
\begin{equation*}
  \tag{Modified Diffusion Approximation}
  F_{n,1}(t) ~\approx~
  \frac{1}{d} u_1(d t) - \frac{1}{d^2}u_2(d t)+\frac{1}{\sqrt{n d}} Z(d t),
\end{equation*}
where $u_1(t)$ and $Z(t)$ are the same as the Diffusion Approximation, and $u_2(t)$ is the fluid limit of $U_{n,2}(t)$ in Theorem~\ref{thm:mainFluid}.

Since $Z$ is a centered Gaussian process, the distribution of $F_{n,1}(t)$ is approximately normal for fixed $t$. To be able to describe the variance, we need $\sigma^2(t)=\var[Z(t)]$. From standard SDE results, $\sigma^2(t)$ satisfies the ODE
\begin{equation}
  \frac{d}{dt} \sigma^2(t) ~=~ -2e^{-u_1(t)}\sigma^2(t)+\lambda+(1-e^{-u_1(t)}),\label{eq:variance}
\end{equation}
with initial condition $\sigma^2(0)=0$.

To investigate the accuracy of our approximations, we collect simulation samples of the fraction of nonempty buffers $F_{n,1}(t)$ and compare the resulting histogram with our approximations. The normal distributions from our two approximations of $F_{n,1}(t)$ have the same variance, but their means are different.

First, we check the accuracy of Diffusion Approximation for moderate $n$ and $d$. For $\lambda=0.7$ and $n=20$, we produce
a histogram with 100000 samples of $F_{20,1}(50)$ for $d=4$ and $d=12$ and compare this with 
the probability density function of the normal distribution from Diffusion Approximation. 
Figure~\ref{fig:diffusion_moderate} shows the results.
\begin{figure}
\centering
\includegraphics[width=\textwidth]{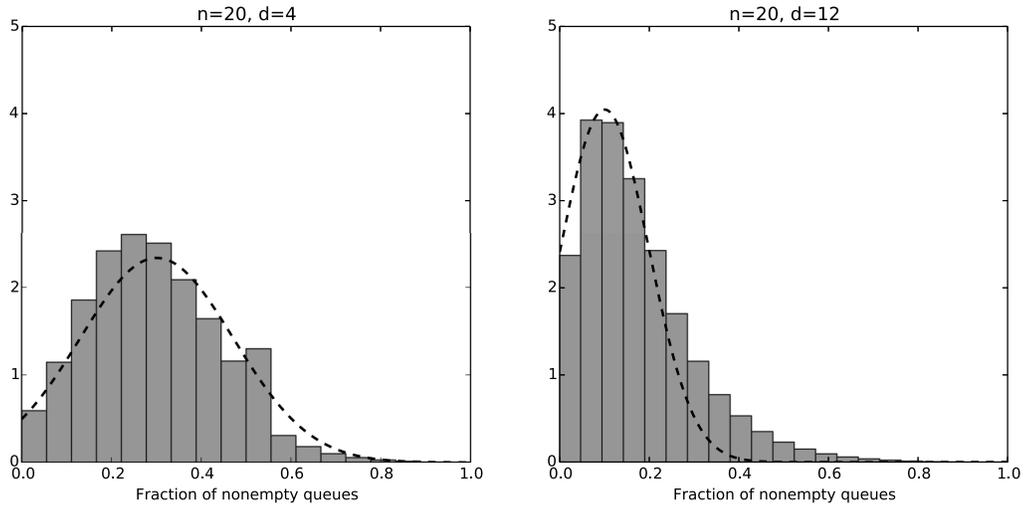}
\caption{Diffusion Approximation versus simulation of the distribution of $F_{n,1}(50)$ for moderate $n$ and $d$. Left: $n=20, d=4$, right: $n=20, d=12$.}
\label{fig:diffusion_moderate}
\end{figure}
Through these and other experiments, we find that Diffusion Approximation is accurate even when $n$ is moderate and it works best in cases where $d$ is small compared to $n$, which is the regime of our theoretical results. When $d$ is large compared to $n$, then the distribution becomes more concentrated at $0$. 

Second, we verify our approximations for large $n$ and small $d$. 
Applying algorithms with small computational complexity to large systems is most meaningful in practice, and this is the case in our model when the number of buffers $n$ is large and the number of samples $d$ is small. By simulation, we obtain histograms of $1000$ samples of the fraction of nonempty queues at time $50$ ($F_{n,1}(50)$) for $n=1000$ and $\lambda=0.7$ as in Figure~\ref{fig:diffusion_large}. 
This result shows that the ODE \eqref{eq:variance} gives a good approximation of the variance of $F_{n,1}(50)$. 
For the mean of $F_{n,1}(50)$, Modified Diffusion Approximation is more accurate than Diffusion Approximation when $d$ is relatively small.
As $d$ grows, Diffusion Approximation better estimates the mean of $F_{n,1}(50)$. This shows that our theorems provide good approximations in practically attractive situations.
\begin{figure}
\centering
\includegraphics[width=\textwidth]{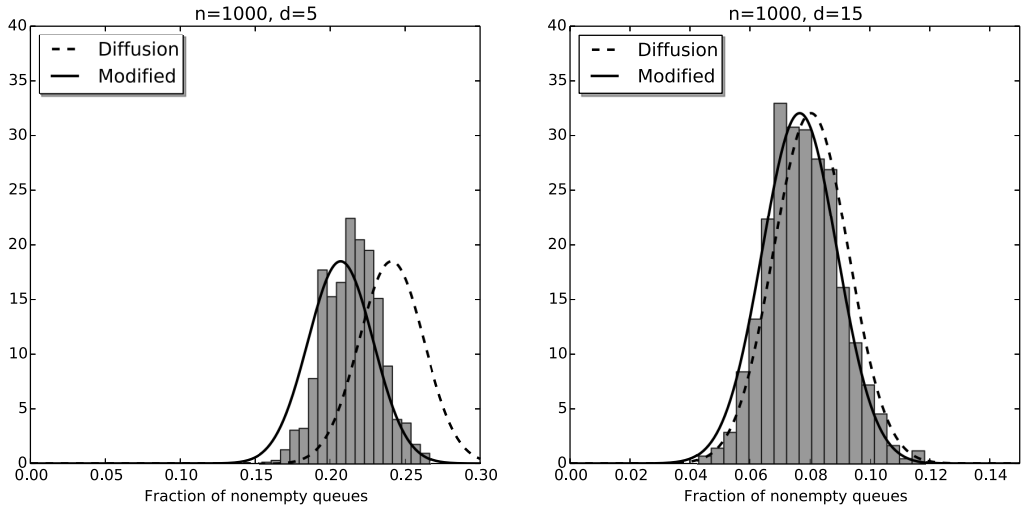}
\caption{Our approximations versus simulation of the distribution of $F_{n,1}(50)$ for large $n=1000$. Left: $d=5$, right: $d=15$. Dash lines are from Diffusion Approximation and solid lines are based on Modified Diffusion Approximation.}
\label{fig:diffusion_large}
\end{figure}
    
We next empirically study when our approximation works well, with the objective to find a criterion depending on $n$, $d$, and $\lambda$ for the validity of our approximation. 
From the Modified Diffusion Approximation, we find the following approximations for the mean and the standard deviation of $F_{n,1}(t)$ for reasonably large $t$:
\begin{equation*}
  \mu ~\simeq~ \left(\frac{1}{d}-\frac{\lambda}{d^2}\right)\log\left(\frac{1}{1-\lambda}\right),\quad
  \sigma ~\simeq~ \frac{1}{\sqrt{nd}}\frac{\lambda}{1-\lambda},
\end{equation*}
where we use Proposition~\ref{prop:fluidLimit} and we set $d\sigma^2(t)/dt=0$ in (\ref{eq:variance}).

We use the Kolmogorov-Smirnov distance between our approximation and the empirical distribution (from simulation) as a measure of accuracy of our approximation.
We find that the quality of our approximation depends on $n$, $d$, and $\lambda$ mostly through $\mu$ and $\sigma$, and 
Figure~\ref{fig:KS_test} summarizes the data from our experiments by plotting the results in the $(\mu,\sigma)$ plane.
The experiments show that the Modified Diffusion Approximation works well if $\mu$ and $\sigma$ satisfy
$\sigma< \mu/3$ and $\sigma>2(\mu-1/4)/3$.
We have also tested the choice of $t$ on the accuracy of our approximation, and we found that it does not have a 
significant effect.
\begin{figure}
\centering
\includegraphics[width=3in]{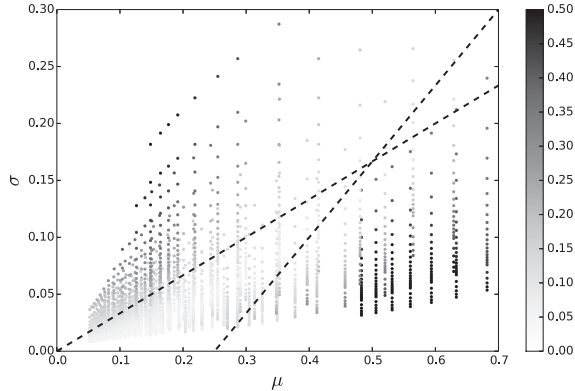}
\caption{The Kolmogorov-Smirnov test statistic for various parameter values.
We use 5000 simulation replications to estimate the distribution of $F_{n,1}(100)$ for
$n=100,150,\ldots,1000, 1200,\ldots,2000$,
$d=2,5,7,10,12,\dots,30$, and $\lambda=0.80,0.82,0.84,\dots,0.98,0.99$.}
\label{fig:KS_test}
\end{figure}

Another observation we get from these simulation experiments is that the variance is not negligible compared to the mean of the fraction of nonempty queues even when $n$ is large. Existing literature exclusively focuses on the performance of algorithms in the mean-field large-buffer regime with the fluid limit, but our experiments highlight that the second-order approximation is also important. Our work is the first investigation in this direction. 

\subsection{Performance vs. complexity}
To see the trade-off between performance and complexity, we measure the complexity and performance
through CPU-time and average queue length, respectively. For a system with $n$ buffers where the server samples $d$ buffers, the CPU-time consumed during a fixed time is $O(d n)$ and our fluid limit theorem concludes that the average queue length is proportional to $1/d$. 

For a fixed number $n$ of buffers in the system, we simulate systems with varying number of sampled buffers $d$. We run our simulation up to time $t=50$ with $\lambda=0.7$ and measure the CPU-time consumption and the average queue length at $t=50$ for $1000$ samples of each case.  The results of our experiments are represented graphically in Figure~\ref{fig:tradeoff_1}. 

Figure~\ref{fig:tradeoff_1} shows that CPU-time per buffer (computational complexity) is indeed proportional to the number of sampled buffers $d$, and that the average queue length (performance) is inverse-proportional to the sample size $d$. 
Therefore, the simulation study confirms our theoretical results on the quantitative trade-off between performance and complexity.
\begin{figure}
\centering
\includegraphics[width=2.4in]{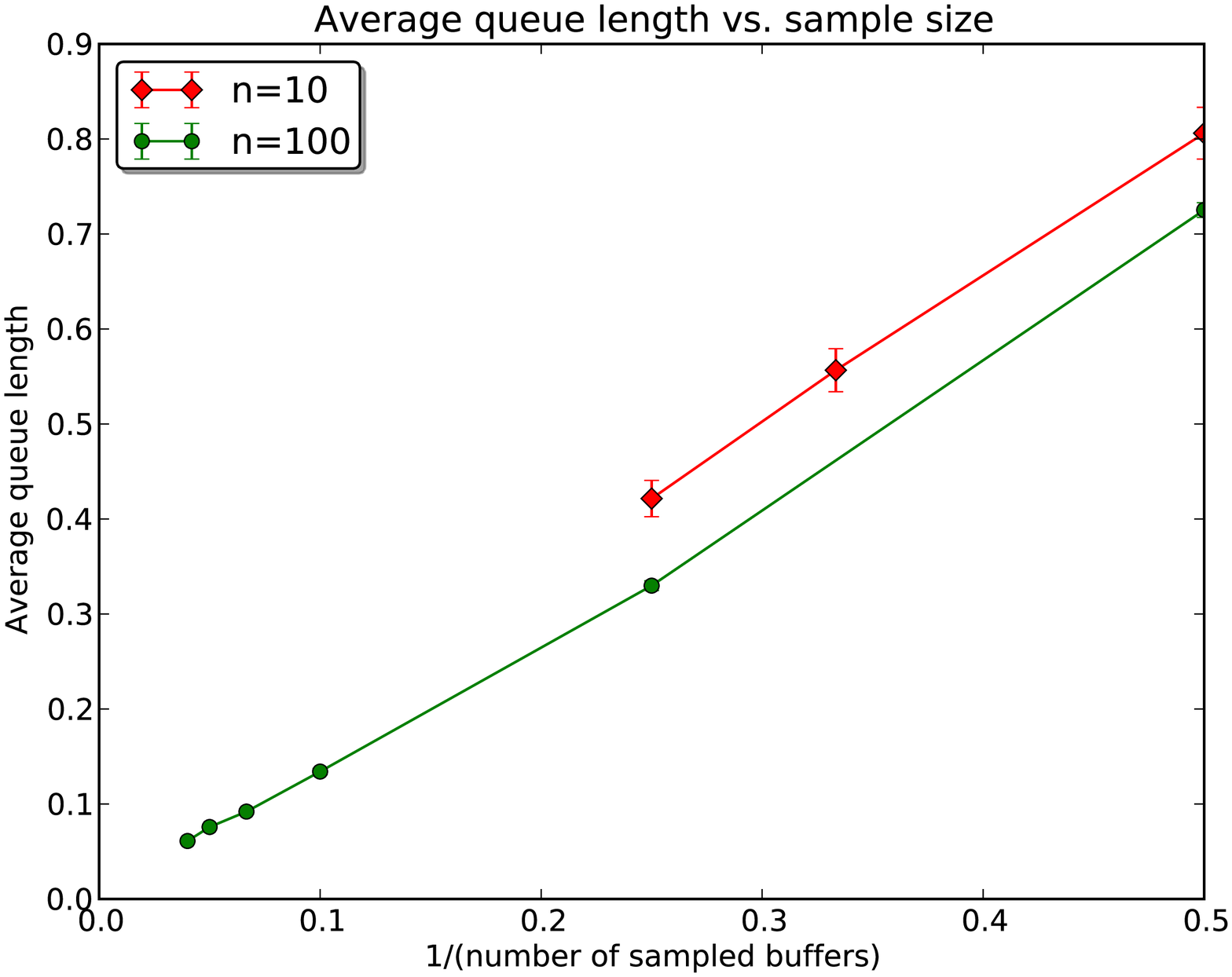}
\includegraphics[width=2.4in]{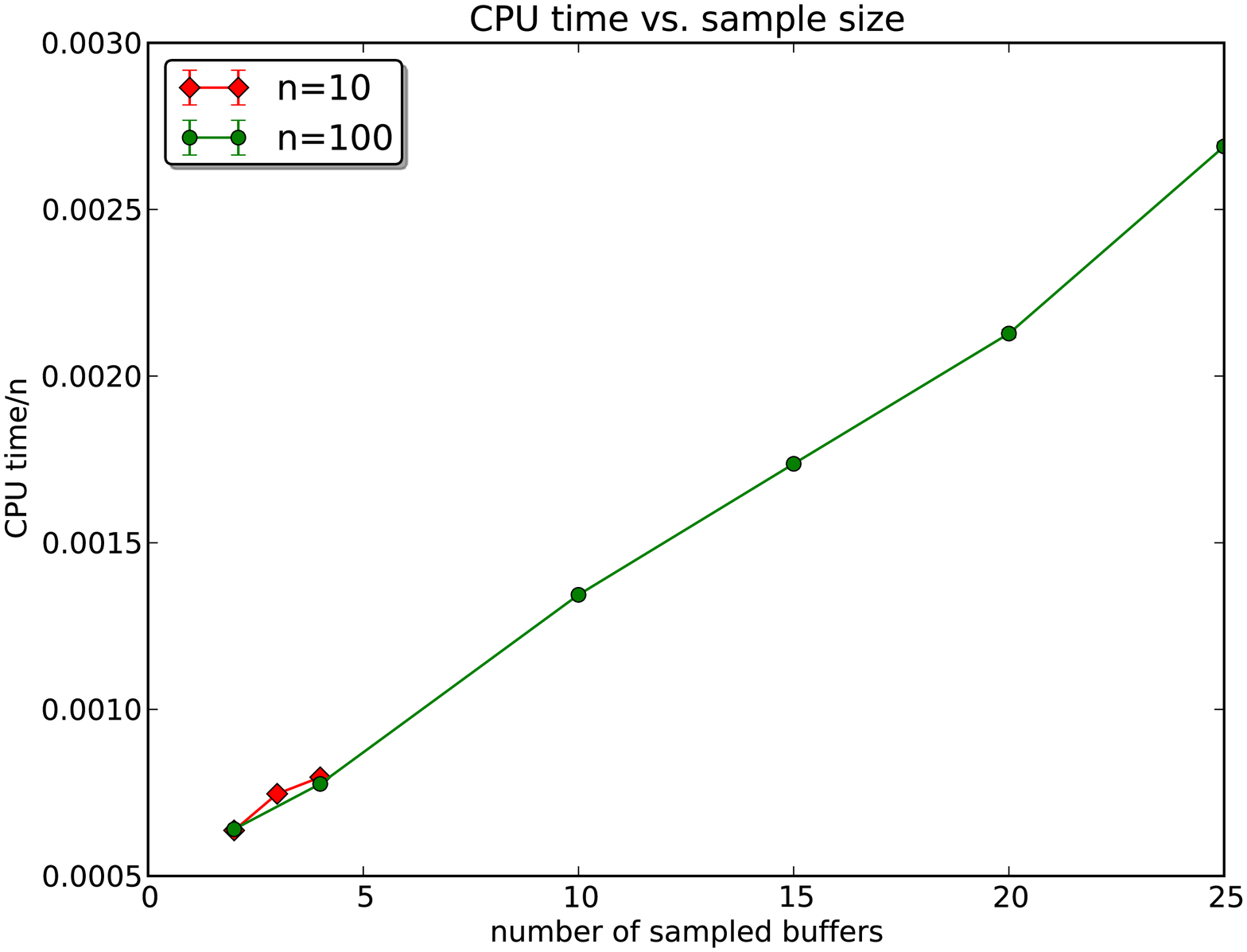}
\caption{Performance versus complexity for $n=10$, $d=2,3,4$ and for $n=100$, $d=2,4,10,15,20,25$.
Left: average queue length vs. sample size $d$. Right: CPU time per buffer vs. sample size $d$.}
\label{fig:tradeoff_1}
\end{figure}

\section{Proofs of the limit theorems}\label{sec:proofs}
This section provides the proofs of the two theorems in Section~\ref{sec:main_results}. Before going into detail, we first introduce the key ideas in the proofs. 

\subsection*{Starting point for the proofs}
We now discuss the starting point of the proofs of our limit theorems, 
particularly focusing on Theorem~\ref{thm:mainFluid}.
Several additional technical tools are needed to fill in the details, and we work these out in Sections~\ref{sec:dynamics_of_the_first_term}--\ref{sec:second_order_approximation}.

Instead of working directly with the random variables $U_{n,k}$, as in \cite{Tsitsiklis:wz}  
we rely on the auxiliary random variables
\begin{equation*}
	V_{n,k}(t)=\sum_{j=k}^\infty F_{n,j}(t),
\end{equation*}
for all $k\geq 0$.

For $k\geq 1,$ $V_{n,k}(\,\cdot\,)$ increases by $1/n$ when there is an arrival in queues with length greater than or equal to $k-1$ 
and it decreases by $1/n$ if the server processes a task in a queue with length greater than or equal to $k$. 
Thus, we have
\begin{equation}\label{eq:exact}
	V_{n,k}(t)=V_{n,k}(0)+\frac{1}{n}\,A_{n,k}\!\left(\lambda\,n\!\int_0^t\!F_{n,k-1}(s)\,\ud s\right)-
	\frac{1}{n}\,S_{n,k}\!\left(n\!\int_0^t \left[1-\left(1-F_{n,k}(s)\right)^{d(n)}\right]\,\ud s\right),
\end{equation}
where $A_{n,k}(\,\cdot\,)$ and $S_{n,k}(\,\cdot\,)$ are independent Poisson processes with rate $1$. 

Upon multiplying (\ref{eq:exact}) by $d(n)^k$ and rescaling time by a factor $d(n)$, we obtain, after substituting $U$ in terms of $F$, 
\begin{eqnarray*}
	d(n)^k\:V_{n,k}\!\left(\frac{t}{d(n)}\right)&=&d(n)^k \:V_{n,k}(0)+\frac{d(n)^k}{n}\:A_{n,k}\!\left(\lambda\,\frac{n}{d(n)^k}\!\int_0^t U_{n,k-1}(s)\,\ud s\right) \\
&&-\frac{d(n)^k}{n}\,S_{n,k}\!\left(\frac{n}{d(n)}\int_0^t \left[1-\left(1-\frac{U_{n,k}(s)}{d(n)^k}\right)^{d(n)}\right]\,\ud s\right).
\end{eqnarray*}
Upon replacing $A_{n,k}$ and $S_{n,k}$ by their law-of-large-numbers approximations (the identity function), we get 
\begin{eqnarray*}
	d(n)^k\:V_{n,k}\!\left(\frac{t}{d(n)}\right)&\approx& d(n)^k\:V_{n,k}(0)+\lambda \int_0^t U_{n,k-1}(s)\,\ud s\\
&&\mbox{}-
	d(n)^{k-1} \int_0^t \left[1-\left(1-\frac{U_{n,k}(s)}{d(n)^k}\right)^{d(n)}\right]\,\ud s,
\end{eqnarray*}
and a similar `second order' representation can be obtained when $A_{n,k}$ and $S_{n,k}$ are replaced by their central limit theorem approximations.
For these approximations to be justified, we need $d(n)^k = o(n)$.
Continuing with the fluid approximation, since $U_0(t)=1$, we obtain for $k=1$,
\[
d(n)^k\:V_{n,1}\!\left(\frac{t}{d(n)}\right)~\approx~ d(n)^k\:V_{n,1}(0)+\lambda\,t-\int_0^t [1-e^{-U_{n,1}(s)}] \,\ud s,
\]
while we obtain for $k\ge 2$,
\[
d(n)^k\:V_{n,k}\!\left(\frac{t}{d(n)}\right)~\approx ~d(n)^k\:V_{n,k}(0)+\lambda \int_0^t U_{n,k-1}(s)\,\ud s- \int_0^t U_{n,k}(s)\,\ud s.
\]

Next we use the following relation between $V_{n,k}(t)$ and $U_{n,k}(t)$:
\begin{equation}
\label{eq:UintermsofV}
	U_{n,k}(t)~=~ d(n)^{k}\;V_{n,k}\!\left(\frac{t}{d(n)}\right)-d(n)^{k}\;V_{n,k+1}\!\!\left(\frac{t}{d(n)}\right).
\end{equation}
The second term on the right-hand side of (\ref{eq:UintermsofV}) vanishes on the fluid scale, but it has to be taken into account
on the diffusion scale. 

The above outline is formalized through a mathematical induction argument.
The next section is devoted to the induction base for the fluid limit theorem, $k=1$. 
Section~\ref{sec:higher_order_dynamics} considers the induction hypothesis for the fluid limit theorem.
Section~\ref{sec:second_order_approximation} addresses the proof of the diffusion limit theorem.

\subsection{Fluid limit: dynamics of the first term} 
\label{sec:dynamics_of_the_first_term}

In this section, we prove the base of the induction by showing the existence of the fluid limit of $U_{n,1}(t)$ and finding the dynamics of the limit. 
The strategy of the proof is the following:
\begin{enumerate}
	\item[1.] The proof evolves around the evolution of $d(n)\,V_{n,1}\!\left(t/d(n)\right)$ and $d(n)\,V_{n,2}\!\left(t/d(n)\right)$. By definition, we have 
	\begin{equation}\label{eq:baseRelation}
	U_{n,1}(t)~=~d(n)\,F_{n,1}\!\left(\frac{t}{d(n)}\right)~=~d(n)\,V_{n,1}\!\left(\frac{t}{d(n)}\right)-d(n)\,V_{n,2}\!\left(\frac{t}{d(n)}\right).
	\end{equation}
	\item[2.] We prove in Lemma~\ref{lemma:baseLemma1} that $d(n)\,V_{n,2}\!\left(t/d(n)\right)$ converges (in an appropriate sense) to the zero function. We then prove in Lemma \ref{lemma:baseLemma2}  that $d(n)\,V_{n,1}\!\left(t/d(n)\right)$ has a fluid limit.
A key tool in the latter is Lemma \ref{lemma:kuang} from the appendix, which requires showing that $d(n)\,V_{n,1}\!\left(t/d(n)\right)$ 
 is Lipschitz in some asymptotic sense.
	\item[3.] We deduce from \eqref{eq:baseRelation} that the fluid limits of $U_{n,1}(t)$ and $d(n)\,V_{n,1}\!\left(t/d(n)\right)$ are the same.
	Using \eqref{eq:exact} and the approach outlined in the previous section, we then formulate the differential equation satisfied by the fluid limit.
\end{enumerate}

First, we prove that $d(n)\,V_{n,2}\!\left(t/d(n)\right)$ converges to $0$ uniformly on compact sets for appropriate initial conditions. In particular, it has a fluid limit.
\begin{lemma}\label{lemma:baseLemma1}
	Consider a sequence of systems indexed by $n.$ Assume that $\lim_{n\to\infty} d(n)\,V_{n,2}(0)=0$ and that $\lim_{n\to\infty} F_{n,1}(0)=0.$ Then we have
	\begin{equation*}
		\lim_{n\to\infty} d(n)\,V_{n,2}\!\left(\frac{t}{d(n)}\right) = 0,
	\end{equation*}
	uniformly on compact sets, almost surely.
\end{lemma}
\begin{proof}
	Let $W_n(\,\cdot\,)$ be the process which increases by $1$ whenever there is an arrival, a service completion, or the end of a wasted service in the $n$th system. Note that $W_n(\,\cdot\,)$ is a Poisson process with rate $(1+\lambda)n.$ For any $t>0,$ the total number of increases of $F_{n,1}(\,\cdot\,)$ in $(0,t\,]$ is less than or equal to $W_n(t).$ Since $F_{n,1}(\,\cdot\,)$ increases by $1/n$ at a time, we obtain, for $t>0$,
	\begin{equation*}
		0~\leq~ F_{n,1}\!\left(\frac{t}{d(n)}\right) ~\leq~ F_{n,1}(0)+\frac{1}{n}\,W_n\!\left(\frac{t}{d(n)}\right),
	\end{equation*}
	By our assumption on $F_{n,1}(0)$ and Lemma \ref{lemma:flln}, 
$F_{n,1}(t/d(n))$ thus converges almost surely to $0$ as $n\to\infty$, uniformly on compact sets.
	From \eqref{eq:exact}, we also deduce that
	\begin{eqnarray*}
		d(n)\,V_{n,2}\!\left(\frac{t}{d(n)}\right)
		&\leq& d(n)\,V_{n,2}(0)+\frac{d(n)}{n}\,A_{n,2}\!\left(\lambda\, n \int_0^{t/d(n)} F_{n,1}(s)\,\ud s\right)\\
		&=& d(n)\, V_{n,2}(0)+\frac{d(n)}{n}\,A_{n,2}\!\left(\frac{\lambda\, n}{d(n)} \int_0^{t} F_{n,1}\!\left(\frac{s}{d(n)}\right)\,\ud s\right).
	\end{eqnarray*}
	Upon applying Lemma~\ref{lemma:integral}, Lemma~\ref{lemma:flln}, and Lemma~\ref{lemma:randomTime},  the second term converges almost surely to $0$ as $n\to\infty$, uniformly on compact sets.
The claim thus follows from the assumption on $V_{n,2}(0)$.
\end{proof}

In the next lemma,
we prove that, almost surely, $d(n)\,V_{n,1}\!\left(t/d(n)\right)$ satisfies the assumptions of Lemma \ref{lemma:kuang}, i.e., that 
it is Lipschitz in some asymptotic sense. This is a key ingredient in establishing the 
existence of the fluid limit of $d(n)\,V_{n,1}\!\left(t/d(n)\right)$.

\begin{lemma}\label{lemma:baseLemma2}
	Consider a sequence of systems indexed by $n.$ Assume that there is some $v\in\Rea_+$ such that
	\begin{equation*}
		\lim_{n\to\infty} d(n)\,V_{n,1}(0)=v.
	\end{equation*}
	Then any subsequence of $\left\{d(n)\,V_{n,1}\!\left(t/d(n)\right)\right\}_{n\in\Nat}$ has a subsequence that converges to a Lipschitz function uniformly on compact sets, almost surely. 
\end{lemma}

\begin{proof}
	Fix $T>0$, and recall the construction of the Poisson process $W_n(\,\cdot\,)$ with rate $(1+\lambda) n$ 
from the proof of Lemma~\ref{lemma:baseLemma1}.
For $a,b\in [0,T]$ with $a<b$, the total number of increases or decreases of $V_{n,1}(t)$ in $(a,b\,]$ is less than or equal to $|W_n(a)-W_n(b)|$.
	Since $d(n)\,V_{n,1}(\,\cdot\,)$ increases or decreases by $d(n)/n$ at a time, there exists some $\gamma_n=\gamma_n(T)$ such that $\lim_{n\to\infty}\gamma_n=0$ almost surely and
	\begin{eqnarray*}
		\left|d(n)\,V_{n,1}\!\left(\frac{a}{d(n)}\right)-d(n)\,V_{n,1}\!\left(\frac{b}{d(n)}\right)\right|
		&\leq& 2\left|\frac{d(n)}{n}\,W_n\!\left(\frac{a}{d(n)}\right)-\frac{d(n)}{n}\,W_n\!\left(\frac{b}{d(n)}\right)\right| \\
		&\leq& 2(1+\lambda)|a-b|+\gamma_n.
	\end{eqnarray*}
	 By Lemma~\ref{lemma:kuang}, any subsequence of $\left\{d(n)\,V_{n,k}\!\left(t/d(n)\right)\right\}_{n\in\Nat}$ has a subsequence that converges to a $2(1+\lambda)$-Lipschitz function uniformly on $[0,T]$,  almost surely.
\end{proof}
	
With \eqref{eq:baseRelation} and the preceding lemmas, we can prove that any subsequence of $\left\{U_{n,1}(t)\right\}_{n\in\Nat}$ has a convergent subsequence which converges to a Lipschitz function $u(t)$.
In the next proposition, we prove that the limit is independent of the subsequence, so that convergence of $\left\{U_{n,1}(t)\right\}_{n\in\Nat}$ to $u(t)$ on compact sets follows.

\begin{proposition}\label{prop:firstTerm}
	Consider a sequence of systems indexed by $n.$ Suppose that for some $v\in\Rea_+$,
	\begin{equation*}
		\lim_{n\to\infty} d(n)\, V_{n,1}(0)=v,\quad \lim_{n\to\infty} d(n)\, V_{n,2}(0)=0,
	\end{equation*}
	almost surely. Then there exists a Lipschitz function $u:[0,\infty)\to\Rea_+$ such that, almost surely,
	\begin{equation*}
		\lim_{n\to\infty} U_{n,1}(t)=u(t),
	\end{equation*}
	uniformly on compact sets and $u$ is the unique solution to the differential equation
	\begin{equation*}
		u'(t)=e^{-u(t)}-(1-\lambda)
	\end{equation*}
	with initial value $u(0)=v.$
	Also, almost surely,
	\begin{equation*}
		\lim_{n\to\infty} d(n)\, V_{n,2}\!\left(\frac{t}{d(n)}\right)=0,
	\end{equation*}
	uniformly on compact sets.
\end{proposition}

\begin{proof}
	By the existence of the limit of $d(n)\,V_{n,1}(0),$ we have $\lim_{n\to\infty} F_{n,1}(0)=0.$ 
	Consider the sequence of bivariate random processes $\left\{\left(d(n)\,V_{n,1}\!\left(t/d(n)\right),U_{n,1}(t)\right)\right\}_{n\in\Nat}.$ 
	From \eqref{eq:baseRelation} and the preceding two lemmas, any subsequence has a subsequence which converges uniformly on compact sets, almost surely. Suppose the convergent subsequence converges to $\left(u(t),u(t)\right),$ for some Lipschitz function $u:[0,\infty)\to\Rea$.  
	
 	We obtain from \eqref{eq:exact} that
	\begin{eqnarray*}
		\lefteqn{d(n)\,V_{n,1}\!\left(\frac{t}{d(n)}\right)}\\
		&=& d(n)\,V_{n,1}(0)+\frac{d(n)}{n}\,A_{n,1}\!\!\left(\lambda n\,\int_0^{t/d(n)} 1 \,\ud s \right) \\
		&&\mbox{}- \frac{d(n)}{n}\,S_{n,1}\!\!\left(n\int_0^{t/d(n)} \left[1-\left(1-F_{n,1}(s)\right)^{d(n)}\right] \,\ud s \right)\\
		&=& d(n)\,V_{n,1}(0)+\frac{d(n)}{n}\,A_{n,1}\!\!\left(\lambda \frac{n}{d(n)} t\right) \\
		&&\mbox{} -\frac{d(n)}{n}\, S_{n,1}\!\!\left(\frac{n}{d(n)}\int_0^t \left[1-\left(1-\frac{U_{n,1}(t)}{d(n)}\right)^{d(n)}\right]\,\ud s\right).
	\end{eqnarray*}
	Thus, letting $n$ go to infinity along the convergent subsequence, we find that, almost surely, the second term converges to $\lambda t$ uniformly on compact sets by Lemma~\ref{lemma:flln}. Moreover, by 
	Lemma~\ref{lemma:composition_convergence}, Lemma~\ref{lemma:dai}, 
	 Lemma~\ref{lemma:flln}, and Lemma~\ref{lemma:randomTime},
	the last term converges almost surely to $\int_0^t \left(1-e^{-u(s)}\right)\,\ud s,$ uniformly on compact sets. Therefore $u(t)$ satisfies the integral equation
	\begin{equation*}
		u(t)=v+ \lambda\,t + \int_0^t \left(1-e^{-u(s)}\right) \,\ud s.
	\end{equation*}
	Since $u$ is absolutely continuous, $u$ is differentiable almost everywhere. If $u(t)$ is differentiable at $t$, we obtain
	\begin{equation}
		u'(t)=e^{-u(t)}-(1-\lambda). \label{eq:firstTermEquation}
	\end{equation}
	By standard existence and uniqueness theorems for ordinary differential equations, 
	there is a unique solution $u:[0,\infty)\to\Rea_+$ satisfying the above differential equation \eqref{eq:firstTermEquation} with initial condition $u(0)=v.$ 
	Thus, every subsequence of $\left\{U_{n,1}(t)\right\}_{n\in\Nat}$ has a subsequence which converges to the same limit $u(t).$ 
	Therefore, $\left\{U_{n,1}(t)\right\}_{n\in\Nat}$ converges to $u(t)$ uniformly on compact sets, almost surely.
\end{proof}

\subsection{Fluid limit: dynamics of higher terms} 
\label{sec:higher_order_dynamics}

In this section, we state and prove the induction step. Let $k\geq 1$ and assume throughout that $\lim_{n\to\infty} {n}/{d(n)^{k+1}}=\infty.$ 
We work under the induction hypothesis that there exists a Lipschitz continuous function $u_k:[0,\infty)\to\Rea_+$ such that
\begin{align}\label{eq:condition1}
	\lim_{n\to\infty} U_{n,k}(t)=u_k(t),
\end{align}
uniformly on compact sets, almost surely, and 
\begin{align}\label{eq:condition2}
	\lim_{n\to\infty} d(n)^k\,V_{n,k+1}\!\!\left(\frac{t}{d(n)}\right) = 0,
\end{align}
uniformly on compact sets, almost surely. Starting from this hypothesis, we prove the existence of the fluid limit of $U_{n,k+1}(t)$ and characterize it through a differential equation. 

The proof roughly follows the same outline as for the dynamics of the first term in Section~\ref{sec:dynamics_of_the_first_term}, i.e., we first establish the existence of the fluid limits and then use (\ref{eq:exact}) to establish the differential equations they satisfy.
The details, however, are different;
for instance, we must avoid a circular argument for establishing an asymptotic Lipschitz property of $d(n)^{k+1}\,V_{n,k+1}\!\left(t/d(n)\right)$ (Lemma \ref{lemma:highLemma2}), 
an issue that did not arise in Section~\ref{sec:dynamics_of_the_first_term}.

\begin{lemma}\label{lemma:highLemma1}
	Consider a sequence of systems indexed by $n$, for which \eqref{eq:condition1} and \eqref{eq:condition2} hold. 
	Assume that 
	\begin{equation*}
		\lim_{n\to\infty} d(n)^{k+1}\,V_{n,k+2}(0)=0,
	\end{equation*}
	almost surely. Then we have
	\begin{equation*}
		\lim_{n\to\infty} d(n)^{k+1}\,V_{n,k+2}\!\left(\frac{t}{d(n)}\right)=0,
	\end{equation*}
	uniformly on compact sets, almost surely.
\end{lemma}
\begin{proof}
By \eqref{eq:exact}, we have
	\begin{eqnarray*}
		\lefteqn{d(n)^{k+1}\,V_{n,k+2}\!\left(\frac{t}{d(n)}\right)}\\
		&\leq& d(n)^{k+1}\,V_{n,k+2}(0)+\frac{d(n)^{k+1}}{n}\,A_{n,k+2}\!\left(\lambda n \int_0^{t/d(n)} F_{n,k+1}(s) \,\ud s\right)\\
		&=&    d(n)^{k+1}\,V_{n,k+2}(0)+\frac{d(n)^{k+1}}{n}\,A_{n,k+2}\!\left(\lambda \frac{n}{d(n)^{k+1}} \int_0^{t} d(n)^k\,\,F_{n,k+1}\!\left(\frac{s}{d(n)}\right) \,\ud s\right).
	\end{eqnarray*}
	Hypothesis \eqref{eq:condition2} implies that $\lim_{n\to\infty} d(n)^k\, F_{n,k+1}\!\left(t/d(n)\right)=0$ almost surely, uniformly on compact sets. Thus, by Lemma~\ref{lemma:integral}, Lemma~\ref{lemma:flln}, and Lemma~\ref{lemma:randomTime}, we obtain that, almost surely,
	\begin{equation*}
		\lim_{n\to\infty} d(n)^{k+1}\,\,V_{n,k+2}\!\left(\frac{t}{d(n)}\right)=0,
	\end{equation*}
	uniformly on compact sets.
\end{proof}

To show the existence of the fluid limit of $d(n)^{k+1}\,V_{n,k+1}\!\left(t/d(n)\right)$, we need to prove that it is Lipschitz in some asymptotic sense, cf.~Lemma~\ref{lemma:kuang}.
For the case $k=0$, we used a scaled version of a Poisson process $W_n(t)$ to prove this for $d(n) \,V_{n,1}\!\left(t/d(n)\right)$.
However, when $k\geq 1,$ a similar modification of $W_n(t)$ does not work 
for $d(n)^{k+1}\,V_{n,k+1}\!\left(t/d(n)\right)$ since $d(n)^{k+1}\,W_n\!\left(t/d(n)\right)$ diverges for $k>0$.
We resolve this difficulty by partitioning an expression for $d(n)^{k+1}\,V_{n,k+1}\!\left(t/d(n)\right)$ into three parts -- an initial part, an arrival part, and a departure part; see (\ref{eq:exact}).
Assuming the existence of a limit for the initial part,
we then show that the other two parts admit fluid limits.

As we shall see, the arrival part depends on $U_{n,k}(t)$ and the induction hypothesis guarantees its convergence. Thus, the existence of the fluid limit of the arrival part follows immediately.
We cannot directly apply the induction hypothesis for the departure part because it turns out to involve $U_{n,k+1}(t)$, the very
quantity we are trying to establish a fluid limit for.
To circumvent this issue, we show that $U_{n,k+1}(t)$ is locally bounded and this allows us
to show that the departure part is Lipschitz continuous in the sense of Lemma \ref{lemma:kuang}.

\begin{lemma}\label{lemma:highLemma2}
	Consider a sequence of systems indexed by $n,$ for which \eqref{eq:condition1} and \eqref{eq:condition2} hold. 
	Suppose that there exists some $v\in\Rea_+$ such that $\lim_{n\to\infty} d(n)^{k+1}\,V_{n,k+1}(0)=v$, almost surely.
	Then any subsequence of $\left\{d(n)^{k+1}\,V_{n,k+1}\!\left(t/d(n)\right)\right\}_{n\in\Nat}$ has a subsequence which converges almost surely to a Lipschitz continuous function uniformly on compact sets.
\end{lemma}

\begin{proof}
	Fix $T>0$. Decompose $d(n)^{k+1}\,V_{n,k+1}\!\left(t/d(n)\right)$ into three parts as follows:
	\begin{equation*}
		d(n)^{k+1}\,V_{n,k+1}\!\left(\frac{t}{d(n)}\right)=d(n)^{k+1}\,V_{n,k+1}(0)+I_n(t)-D_n(t),
	\end{equation*}
	where $I_n(t)$ and $D_n(t)$ are the total increase and decrease amount of $d(n)^{k+1}\,V_{n,k+1}\!\left(t/d(n)\right)$ by time $t,$ respectively. 

	The almost sure limit of $I_n(t)$ is readily found. Indeed, from \eqref{eq:exact}, we have
	\begin{equation*}
		I_n(t)=\frac{d(n)^{k+1}}{n}\,A_{n,k+1}\!\!\left(\lambda n\int_0^{t/d(n)} F_{n,k}(s) \,\ud s\right)
		=\frac{d(n)^{k+1}}{n}\,A_{n,k+1}\!\!\left(\frac{n}{d(n)^{k+1}}\int_0^t U_{n,k}(s) \,\ud s\right),
	\end{equation*}
	which converges almost surely to $\int_0^t u_{k}(s) \,\ud s$ uniformly on $[0,T]$ by Lemma \ref{lemma:integral} and \ref{lemma:randomTime}. 
	
 	Proving the almost sure limit of $D_n(t)$ is more intricate. We obtain from \eqref{eq:exact} that
	\begin{eqnarray}
		\lefteqn{D_n(t)}  \nonumber\\ 
		&=& \frac{d(n)^{k+1}}{n}\,S_{n,k+1}\!\!\left(n\int_0^{t/d(n)} \left(1-(1-F_{n,k+1}(s))^{d(n)}\right) \,\ud s\right)\nonumber \\
		&=& \frac{d(n)^{k+1}}{n}\,S_{n,k+1}\!\!\left(\frac{n}{d(n)}\int_0^{t} \left(1-(1-F_{n,k+1}(s/d(n)))^{d(n)}\right) \,\ud s\right)\nonumber\\
		&=& \frac{d(n)^{k+1}}{n}\,S_{n,k+1}\!\! \left(\frac{n}{d(n)^{k+1}}\int_0^{t} d(n)^k\left[1-\left(1-\frac{U_{n,k+1}(s)}{d(n)^{k+1}}\right)^{d(n)}\right] \,\ud s\right).\label{eq:DintermsofS}
	\end{eqnarray}
The first step for analyzing this expression is to bound the integrand.
Write $M=\sup_{t\in[0,T]}\int_0^t u_k(s) \,\ud s$ and let $\varepsilon>0$. 
Then, for all $t\in [0,T]$ and large enough $n$, we have
	\begin{equation*}
		U_{n,k+1}(t)~\leq~ d(n)^{k+1}\, V_{n,k+1}\!\left(\frac{t}{d(n)}\right)~\leq~ d(n)^{k+1}\,V_{n,k+1}(0)+ I_n(t)~\leq~ v+M+\varepsilon. 
	\end{equation*}
Thus, for all large enough $n,$ we have almost surely
	\begin{equation*}
		d(n)^{k} \left[1-\left(1-\frac{U_{n,k+1}(t)}{d(n)^{k+1}}\right)^{d(n)}\right] 
		~\leq~ d(n)^{k} \left[1-\left(1-\frac{v+M+\varepsilon}{d(n)^{k+1}}\right)^{d(n)}\right] 
		~\leq~ v+M+2\varepsilon
	\end{equation*}
	for all $t\in[0,T].$
Lemma~\ref{lemma:flln} implies that, almost surely,
\[
\lim_{n\to\infty} \sup_{a,b\in[0,(v+M+2\varepsilon)T]}
\left| \frac{d(n)^{k+1}}{n}\: S_{n,k+1}\!\left(\frac{n}{d(n)^{k+1}}b\right) -
\frac{d(n)^{k+1}}{n}\: S_{n,k+1}\!\left(\frac{n}{d(n)^{k+1}}a\right)-(b-a)\right|=0,
\]
which by (\ref{eq:DintermsofS}) shows that $\lim_{n\to\infty} \gamma_n = 0$ almost surely, where
\[
\gamma_n =\sup_{0\le s<t\le T}\left|
D_n(t)-D_n(s) - \int_s^t d(n)^k   \left[1-\left(1-\frac{U_{n,k+1}(u)}{d(n)^{k+1}}\right)^{d(n)}\right] \,\ud u\right|.
\]
We next note that, for $a,b\in[0,T]$,
	\begin{equation*}
		\left|D_n(a)-D_n(b)\right|~\leq~ (v+M+2\varepsilon)|a-b|+\gamma_n.
	\end{equation*}
	Thus, by Lemma \ref{lemma:kuang}, any subsequence of $\{D_{n,k}(\cdot)\}$ has a subsequence that converges to a Lipschitz continuous function.
	Therefore, any subsequence of $\left\{d(n)^{k+1}\,V_{n,k+1}\!\left(t/d(n)\right)\right\}_{n\in\Nat}$ has a subsequence converging to a Lipschitz continuous function uniformly on $[0,T],$ almost surely.
\end{proof}
  
By the preceding two lemmas, any subsequence of $\left\{U_{n,k+1}(t)\right\}_{n\in\Nat}$ has a subsequence which converges almost surely to a Lipschitz function uniformly on compact sets. We prove the induction step through the same argument used in the induction base.

\begin{proposition}\label{prop:higherTerm}
	Consider a sequence of systems indexed by $n$, for which the induction hypothesis \eqref{eq:condition1} and \eqref{eq:condition2} hold.
	Assume that there exists some $v\in\Rea_+$ such that $\lim_{n\to\infty} d(n)^{k+1}\,V_{n,k+1}(0)=v,$ almost surely and $\lim_{n\to\infty} d(n)^{k+1}\,V_{n,k+2}(0)=0.$ 
	Then the sequence $\left\{U_{n,k+1}(t)\right\}_{n\in\Nat}$ converges almost surely to the unique Lipschitz function $u_{k+1}:[0,\infty)\to\Rea_+$ satisfying 
	\begin{equation*}
		u_{k+1}'(t)=\lambda\,u_{k}(t) - u_{k+1}(t),
	\end{equation*}
	with $u(0)=v$, uniformly on compact sets. Moreover, we have 
	\begin{equation*}
		\lim_{n\to\infty} d(n)^{k+1}\,\,F_{n,k+2}\!\left(\frac{t}{d(n)}\right) =0,
	\end{equation*}
	uniformly on compact sets.
\end{proposition}
\begin{proof}
		Consider the sequence of coupled random processes $\left\{\left(d(n)^{k+1}\, V_{n,k+1}\!\left(t/d(n)\right),U_{n,k+1}(t)\right)\right\}_{n\in\Nat}.$ By the preceding lemmas, any subsequence has a subsequence which converges uniformly on compact sets, almost surely. Moreover, the convergent subsequence converges to $(u_{k+1}(t),u_{k+1}(t))$ for some Lipschitz function $u_{k+1}(t)$.

	We deduce from \eqref{eq:exact} that
	\begin{eqnarray*}
		\lefteqn{d(n)^{k+1}\,V_{n,k+1}\!\left(\frac{t}{d(n)}\right)}\\
		&=& d(n)^{k+1}\,V_{n,k+1}(0) + \frac{d(n)^{k+1}}{n} A_{n,k+1}\left(\lambda n \int_0^{t/d(n)}F_{n,k}(s) \,\ud s\right)\\
		&&\mbox{} -\frac{d(n)^{k+1}}{n}\, S_{n,k+1}\!\! \left( n\int_0^{t/d(n)} \left(1-(1-F_{n,k+1}(s))^{d(n)}\right)\,\ud s\right)\\
		&=& d(n)^{k+1}\,V_{n,k+1}(0) + \frac{d(n)^{k+1}}{n}\,A_{n,k+1}\!\!\left(\frac{\lambda n}{d(n)^{k+1}}\int_0^t U_{n,k}(s)\,\ud s\right)\\
		&&\mbox{} -\frac{d(n)^{k+1}}{n}\,S_{n,k+1}\!\!\left(\frac{n}{d(n)^{k+1}}\int_0^{t} d(n)^{k}\left(1-\left(1-\frac{U_{n,k+1}(s)}{d(n)^{k+1}}\right)^{d(n)}\right)\,\ud s\right).
	\end{eqnarray*}
	From	Lemma~\ref{lemma:composition_convergence}, Lemma~\ref{lemma:dai}, Lemma~\ref{lemma:flln}, and Lemma~\ref{lemma:randomTime}, by taking the limit as $n\to\infty$ along the convergent subsequence, we conclude that $u_{k+1}(t)$ satisfies 
	\begin{equation*}
		u_{k+1}(t)=v+ \lambda \int_0^t u_{k}(s)\,\ud s - \int_0^t u_{k+1}(s) \,\ud s.
	\end{equation*}
	Since $u_{k+1}(t)$ is absolutely continuous, $u_{k+1}(t)$ is differentiable almost everywhere. If $u_{k+1}(t)$ is differentiable at $t$, we obtain
	\begin{equation}
		u_{k+1}'(t)=\lambda\,u_k(t) - u_{k+1}(t). \label{eq:highTermEquation}
	\end{equation}
	
	Since the differential equation \eqref{eq:highTermEquation} is linear with inhomogeneous term $\lambda u_k(t)$, it uniquely determines $u_{k+1}(t)$. Thus, every sequence of $\left\{U_{n,k+1}(t)\right\}_{n\in\Nat}$ has a subsequence that converges to the same limit $u_{k+1}(t)$. Therefore,
	$U_{n,k+1}(t)$ converges to $u_{k+1}(t)$ uniformly on compact sets, almost surely.
	
	The last statement of the proposition follows from Lemma~\ref{lemma:highLemma1}.
\end{proof}

Using Proposition~\ref{prop:firstTerm} and Proposition~\ref{prop:higherTerm}, we are now ready to prove our fluid limit theorem.

\begin{proof}[Proof of Theorem~\ref{thm:mainFluid}]
  From the assumptions of Theorem~\ref{thm:mainFluid}, we have
  \begin{equation*}
    \lim_{n\to\infty} U_{n,1}(0)=v_1
  \end{equation*}
  and
  \begin{equation*}
    \lim_{n\to\infty} d(n)\,V_{n,2}(0)=\lim_{n\to\infty}\left( \frac{U_{n,2}(0)}{d(n)}+\dots+\frac{U_{n,K}(0)}{d(n)^{K-1}}+\frac{d(n)^K(F_{n,K+1}(0)+\cdots)}{d(n)^{K-1}}  \right)=0.
  \end{equation*}
  Therefore, Proposition~\ref{prop:firstTerm} yields the fluid limit for $U_{n,1}(t)$, which is
\eqref{eq:condition1} for $k=1$. Lemma~\ref{lemma:baseLemma1} yields \eqref{eq:condition2} for $k=1$.
  
  We next assume that conditions \eqref{eq:condition1} and \eqref{eq:condition2} hold. 
The assumptions in Proposition~\ref{prop:higherTerm} hold because of the assumptions from Theorem~\ref{thm:mainFluid}, as can be seen with a similar argument as above. Thus, Proposition~\ref{prop:higherTerm} and Lemma~\ref{lemma:highLemma1} show that \eqref{eq:condition1} and \eqref{eq:condition2} hold, respectively, with $k$ replaced by $k+1$.
This induction argument establishes Theorem~\ref{thm:mainFluid}.
\end{proof}

\subsection{Diffusion limit} 
\label{sec:second_order_approximation}
In this section, we prove our second limit theorem, Theorem~\ref{thm:mainDiffusion},
a diffusion limit of $U_{n,1}(t)$. To this end, we introduce a new sequence of stochastic processes with the same fluid limit $u_1(t)$ as $\left\{U_{n,1}(t)\right\}_{n\in\Nat}$. For this new sequence, we can apply a result from Kurtz~\cite{kurtz:1978kq} to obtain its second-order approximation. 
We then compare the new processes with $\left\{U_{n,1}(t)\right\}_{n\in\Nat}$ and show that the difference vanishes.

\begin{proof}[Proof of Theorem~\ref{thm:mainDiffusion}]
  From \eqref{eq:exact}, we have
\begin{eqnarray}
  U_{n,1}(t)&=& -d(n)\,V_{n,2}\!\left(\frac{t}{d(n)}\right)+V_{n,1}(0) \nonumber \\ 
  &&\mbox{}+\frac{d(n)}{n}\,A_{n,1}\!\!\left(\frac{\lambda n}{d(n)}\, t\right) - \frac{d(n)}{n}\,S_{n,1}\!\!\left(\frac{n}{d(n)}\int_0^t \left[1-\left(1-\frac{U_{n,1}(s)}{d(n)}\right)^{d(n)}\right]\,\ud s\right). \label{eq:U_n_1}
\end{eqnarray}
Let $\lim_{n\to\infty} n/d(n)=\infty$ and $\lim_{n\to\infty} n/d(n)^2=0$ and assume that $U_{n,k}(0)$ for all $n$ and $k$, and $v_1\in\Rea_+$ satisfies conditions \eqref{eq:diffusion_condition_1} and \eqref{eq:diffusion_condition_2} in Theorem~\ref{thm:mainDiffusion}.

Define a sequence of stochastic processes $\{\widehat{U}_n(t)\}$ as the unique solution to 
\begin{equation}
  \widehat{U}_n(t)~=~v_1+\frac{d(n)}{n}\, A_{n,1}\!\!\left(\frac{n}{d(n)}\int_0^t f_{n,1}( \widehat{U}_n(s))\,\ud s\right) - \frac{d(n)}{n}\,S_{n,1}\!\!\left(\frac{n}{d(n)}\int_0^t f_{n,-1}( \widehat{U}_n(s))\,\ud s \right), \label{eq:tildeU}
\end{equation}
where $f_{n,1}=\lambda$ and 
\begin{equation*}
  f_{n,-1}(x)=\begin{cases} 1-\left(1-\frac{x}{d(n)}\right)^{d(n)} & \textrm{if $0\leq x\leq d(n)$}\\
 1-e^{-x}+e^{-d(n)}& \textrm{otherwise}
 \end{cases}.
\end{equation*}
The process $\widehat{U}_n(t)$ is coupled with $U_{n,1}(t)$. 
We next argue that $\widehat{U}_n(t)$ has a fluid and diffusion approximation prescribed by
the theory developed by Kurtz \cite{kurtz:1978kq} (see Lemma~\ref{lemma:kurtz} in the Appendix).  
Note that the index in \cite{kurtz:1978kq} is $N=n/d(n)$ and $n$ can 
often also be expressed in terms of $N$. This cannot always be done,
but we suppress the arguments needed to deal with such cases.

Let $f_{1}(x)=\lambda$ and $f_{-1}(x)=1-e^{-x}$. After noting that the maximum of $m \left( e^{-x}-(1-x/m)^m \right)$ over $0\leq x\leq m$ converges to $2$ as $m\to\infty$, we have, for large enough $n$,
\begin{equation*}
  \left| f_{n,-1}(x)-f_{-1}(x) \right|~\leq~ \frac{3}{d(n)} ~\leq~ 3\frac{d(n)}{n}.
\end{equation*}
Thus all conditions from Lemma~\ref{lemma:kurtz} are satisfied and $\widehat{U}_{n}(t)$ converges almost surely to $u_1(t)$ uniformly on compact sets, and we have the second-order approximation of $\widehat{U}_n(t)$ such that 
\begin{equation}
  \sqrt{\frac{n}{d(n)}} \left( \widehat{U}_n(t)-u_1(t) \right) ~\Rightarrow~ Z(t),
\label{eq:second_tilde}
\end{equation}
where $Z(t)$ satisfies 
\[
  Z(t)~=~\sqrt{\lambda} B^{(1)}(t)-\int_0^t \sqrt{1-e^{-u_1(s)}} \,\ud B^{(2)}(s) -\int_0^t e^{-u_1(s)}Z(s) \,\ud s 
\]
for independent Wiener processes $B^{(1)}(t)$ and $B^{(2)}(t)$.
We note that the results in \cite{kurtz:1978kq} yield strong approximations; here we only 
use weaker results of convergence in distribution. 

We next compare $U_{n,1}(t)$ with $\widehat{U}_n(t)$
and show that $\sqrt{n/d(n)} \,|U_{n,1}(t)-\widehat{U}_n(t)|\Rightarrow 0$.  
Fix some $T>0$.
From \eqref{eq:U_n_1} and \eqref{eq:tildeU}, we have, since $0\leq U_{n,1}(t)\leq d(n)$ and $f_{n,-1}(t)$ is $1$-Lipschitz continuous,
\begin{eqnarray*}
  \lefteqn{\sqrt{\frac{n}{d(n)}}\left| U_{n,1}(t)-\widehat{U}_n(t) \right|} \\
  &\leq& \sqrt{\frac{n}{d(n)}}\left( V_{n,1}(0)-v_1\right) + \sqrt{n d(n)}\,\,V_{n,2}\!\left(\frac{t}{d(n)}\right) \\
  && \mbox{} + \left|~\widetilde{S}_n\!\left(\int_0^t f_{n,-1}( {U}_{n,1}(s))\,\ud s\right)-\widetilde{S}_n\!\left(\int_0^t f_{n,-1}( \widehat{U}_n(s))\,\ud s\right)~\right| \\
  && \mbox{} + \sqrt{\frac{n}{d(n)}}  \int_0^t \left|~ f_{n,-1}(U_{n,1}(s)) - f_{n,-1}( \widehat{U}_n(s)) ~\right| ds \\
  &\leq& \varepsilon_n(t) + \int_0^t \sqrt{\frac{n}{d(n)}}\:\left|~ U_{n,1}(s) - \widehat{U}_n(s) ~\right| \,\ud s,
\end{eqnarray*}
where 
\begin{equation*}
  \widetilde{S}_n(t) ~=~ \sqrt{\frac{n}{d(n)}}\,\left( \frac{d(n)}{n}\,\, S_{n,1}\!\left(\frac{n}{d(n)}t\right) -t\right)
\end{equation*}
and
\begin{eqnarray*}
  \varepsilon_n(t) &=& \sqrt{\frac{n}{d(n)}}\left( V_{n,1}(0)-v_1\right)+\sqrt{n\, d(n)}\,\,V_{n,2}\!\left(\frac{t}{d(n)}\right) \\
  && \mbox{} + \left|~\widetilde{S}_n\!\left(\int_0^t f_{n,-1}(U_{n,1}(t))\,\ud s\right)-\widetilde{S}_n\!\left(\int_0^t f_{n,-1}(\widehat U_{n,1}(t))\,\ud s\right)~\right|.
\end{eqnarray*}
By Gronwall's inequality, we obtain, for $t\in[0,T]$,
\begin{equation*}
  \sqrt{\frac{n}{d(n)}}\,\left|~ U_{n,1}(t)-\widehat{U}_n(t)~\right|
  ~\leq~ \varepsilon_n(t) + e^t \int_0^t \varepsilon_n(t)\, \ud s 
  ~\leq~ L \cdot \sup_{t\in [0,T]} \varepsilon_n(t),
\end{equation*}
where $L=1+T e^T$.

We proceed by showing that $\varepsilon_n(t)\Rightarrow 0$.
From \eqref{eq:exact}, we find that
\begin{equation*}
  \sqrt{n\,d(n)}\,\,V_{n,2}\!\left(\frac{t}{d(n)}\right) 
  ~\leq~ \sqrt{n\,d(n)}\,V_{n,2}(0) + \sqrt{n\,d(n)}\:S_{n,2}\!\left(\frac{n}{d(n)^2} \int_0^t U_{n,1}(s) \,\ud s\right),
\end{equation*}
which converges to $0$ almost surely as $n\to\infty$ uniformly on compact sets, by \eqref{eq:diffusion_condition_2}, Lemma~\ref{lemma:integral}, Lemma~\ref{lemma:randomTime} with $\lim_{n\to\infty} n/d(n)^2=0$.
Also, from Lemma~\ref{lemma:fclt} and Lemma~\ref{lemma:randomTime}, we deduce that
\begin{eqnarray*}
  \left( \widetilde{S}_n\left( \int_0^t f_{n,-1}(U_{n,1}(s))\,\ud s \right),
    \widetilde{S}_n\left( \int_0^t f_{n,-1}( \widehat{U}_n(s))\,\ud s \right) \right)\\
  \Rightarrow
  \left(  B\left( \int_0^t \left[1-e^{-u_1(s)}\right] ds \right),
    B\left( \int_0^t \left[1-e^{-u_1(s)}\right] ds \right) \right),
\end{eqnarray*}
as $n\to\infty$, where $B$ is a standard Wiener process. 
By the continuous mapping theorem, we conclude that, as $n\to\infty$,
\begin{equation*}
  \varepsilon_n(t)~\Rightarrow~ 0,
\end{equation*}
and therefore 
\begin{equation*}
  \sqrt{\frac{n}{d(n)}}\left( U_{n,1}(t)-\widehat{U}_n(t)\right) ~\Rightarrow~ 0.
\end{equation*}
From \eqref{eq:second_tilde}, we conclude that, as $n\to\infty$,
\begin{equation*}
  \sqrt{\frac{n}{d(n)}} \left( U_{n,1}(t)-u_1(t) \right) ~\Rightarrow~ Z(t),
\end{equation*}
as claimed.
\end{proof}

\section*{Acknowledgments}
This research is supported in part by NSF grant CMMI-1252878.
We thank Ilyas Iyoob for fruitful discussions.

\appendix
\section{Appendix} 
\label{sec:appendix}
This appendix reviews elements of convergence theory of functions and stochastic processes. 

For fixed $T>0$, $D^k[0,T]$ is the space of functions from $[0,T]$ to $\Rea^k$ that are right-continuous with left-limits (RCLL) equipped with the norm
\begin{equation*}
  \norm{f}_T ~:=~ \sup_{0\leq t\leq T} \norm{f(t)}_\infty
\end{equation*}
and the associated topology of uniform convergence.
We define $D^k[0,\infty)$ similarly, and we equip it with the product metric (of convergence on compact sets) and its associated topology.

We interpret a stochastic process $X$ in this context as a measurable mapping from a probability space $(\Omega,\mathcal{F},\mathbb{P})$ to $D^k[0,\infty)$. For a sequence $\{X_n\}_{ n\in\Nat}$ 
of stochastic processes and a stochastic process $X$, we say that $\{X_n\}_{n\in\Nat}$ converges almost surely to $X$ 
uniformly on compact sets if 
\begin{equation*}
  \mathbb{P}\left( \lim_{n\to\infty} \norm{X_n-X}_T=0 \right) =1,
\end{equation*}
for all $T>0$. 

For a stochastic process $X$, we can define a probability measure $P_X$ on $D^k[0,T)$ for any $T>0$. We say that a sequence $\{X_n\}_{n\in\Nat}$ of stochastic processes converges in distribution to a stochastic process $X$ if, for all $T>0$,
\begin{equation*}
  \lim_{n\to\infty} \int_{D^k[0,T]} f \,\ud P_{X_n} ~=~ \int_{D^k[0,T]} f \,\ud P_X
\end{equation*}
for every bounded and continuous real-valued function $f$ on $D^k[0,T]$. We abbreviate this by
\begin{equation*}
  X_n ~\Rightarrow~ X,
\end{equation*}
as $n\to\infty$.

The following lemmas contain results about convergence of functions that are needed to prove our theorems. The first three lemmas are basic results about uniform convergence on compact sets. The proof of the third lemma can be found in \cite{Dai1995}.

\begin{lemma}\label{lemma:integral}
Let $\{f_n\}_{n\in\Nat}$ be a sequence of real-valued functions defined on $[0,\infty)$ and assume that it converges to a function  $f:[0,\infty)\to\Rea$ uniformly on compact sets. Assume that the functions $F_n:[0,\infty)\to\Rea$ with $F_n(t)=\int_0^t f_n(s)\,\ud s$ and 
$F:[0,\infty)\to\Rea$ with $F(t)=\int_0^t f(s)\,\ud s$ are well-defined. Then, as $n\to\infty$,
	$\{F_n\}_{n\in\Nat}$ converges to $F$ uniformly on compact sets.
\end{lemma}

\begin{lemma}\label{lemma:composition_convergence}
  Let $\{f_n\}_{n\in\Nat}$ and $\{g_n\}_{n\in\Nat}$ be two sequences of real-valued functions defined on $[0,\infty)$. Assume that $g_n$ is nonnegative.
  If, as $n\to\infty$, $\{f_n\}_{n\in\Nat}$ and $\{g_n\}_{n\in\Nat}$ converges uniformly on compact sets to real-valued functions $f$ and $g$, respectively, and $f$ and $g$ are continuous, then, as $n\to\infty$, the sequence $\{f_n(g_n)\}_{n\in\Nat}$ converges to $f(g)$ uniformly on compact sets.
\end{lemma}
\begin{proof}
  Fix $T>0$ and $\varepsilon>0$. Since $g$ is continuous on $[0,T]$, there exists $M>0$ such that $|g(t)|\leq M$ for all $t\in[0,T]$. 
  Since $f$ is continuous on $[0,M+1]$, there exists $0<\delta<1$ such that, for $s,t\in [0,M+1]$, $|t-s|<\delta$ implies $|f(t)-f(s)|\leq \varepsilon/2$. 
  Let $L=\max\{T,M+1\}$. 
  
  From the fact that $f_n\to f$ and $g_n \to g$ as $n\to\infty$ uniformly on compact sets, there exists some $N\in\Nat$ such that $n\geq N$ implies $|f_n(t)-f(t)|\leq \min\{\varepsilon/2,\delta\}$ and $|g_n(t)-g(t)|\leq \min\{\varepsilon/2,\delta\}$ for all $t\in[0,L]$. Then, for all $t\in[0,T]$ and $n\geq N$, we have
  \begin{equation*}
    |g_n(t)|~\leq~ |g_n(t)-g(t)| +|g(t)|~\leq~ 1+M.
  \end{equation*}
  Thus, if $n\geq N$, we have
  \begin{equation*}
    |f_n(g_n(t))-f(g(t))| 
    ~\leq~ |f_n(g_n(t))-f(g_n(t))| + |f(g_n(t))-f(g(t))| ~<~ \varepsilon,
  \end{equation*}
  for all $t\in[0,T]$. Therefore, $f_n(g_n)$ converges to $f(g)$ as $n\to\infty$ uniformly on compact sets.
\end{proof}

\begin{lemma}\label{lemma:dai}
  Let $\{f_n\}_{n\in\Nat}$ be a sequence of nondecreasing real-valued functions on $[0,\infty)$ and let $f$ be a continuous function on $[0,\infty)$. Assume that $\lim_{n\to\infty} f_n(t)= f(t)$ for all rational numbers $t\in[0,\infty)$. Then $\{f_n\}_{n\in\Nat}$ converges to $f$, as $n\to\infty$, uniformly on compact sets.
\end{lemma}

The next lemmas are the functional law of large numbers and the functional central limit theorem for Poisson processes, see for instance \cite{Chen:2001wa}.

\begin{lemma}[Functional Law of Large Numbers] \label{lemma:flln}
Let $A$ be a Poisson process with rate $\lambda.$ Then, as $n\to\infty$, we have almost surely,
\begin{equation*}
  \frac{1}{n}\,A(n\,t) ~\to~ \lambda t,
\end{equation*}
uniformly on compact sets. Also, if $f(n)=o(n)$ and $\lim_{n\to\infty}f(n)=\infty$, we have almost surely,
\begin{equation*}
  \frac{1}{n}\,A\!\left(\frac{n}{f(n)}\,t\right) ~\to~ 0,
\end{equation*}
as $n\to\infty$, uniformly on compact sets.
\end{lemma}

\begin{lemma}[Functional Central Limit Theorem] \label{lemma:fclt}
  Let $A$ be a Poisson process with rate $1$. Then, as $n\to\infty$,
  \begin{equation*}
    \sqrt{n} \left( \frac{1}{n}\,A(n\,t)-t \right) ~\Rightarrow~ B(t),
  \end{equation*}
  where $B(t)$ is the standard Wiener process.
\end{lemma}

The following lemma is often called the random time-change theorem, see for instance \cite{Chen:2001wa}.

\begin{lemma}[Random Time-Change Theorem]\label{lemma:randomTime}
  Let $\{f_n\}_{n\in\Nat}$ and $\{g_n\}_{n\in\Nat}$ be two sequences in $D^k[0,\infty)$. Assume that each component of $g_n$ is nondecreasing with $g_n(0)=0$. If as $n\to\infty$, $(f_n,g_n)$ converges uniformly on compact sets to $(f,g)$ and $f$ and $g$ are continuous, then
  \begin{equation*}
    \lim_{n\to\infty}f_n(g_n) ~\to~ f(g),
  \end{equation*}
  uniformly on compact sets, where the $i$th component of $f(g)$ is the composition of $i$th component of $f$ and $i$th component of $g$. 
\end{lemma}

The next lemma can be used to show the existence of a fluid limit of a sequence of stochastic processes. Intuitively, it entails that if the fluctuations of a sequence of functions are asymptotically bounded by the fluctuations of a Lipschitz function, then any subsequence has a convergent subsequence which converges to a Lipschitz function. This lemma immediately follows from arguments in Appendix A in \cite{Tsitsiklis:wz}.

\begin{lemma}\label{lemma:kuang}
	Fix $T>0.$ Let $\{f_n\}_{n\in\Nat}$ be a sequence in $D[0,T].$ Assume that $|f_n(0)|\leq M$ and 
	\begin{equation*}
		|f_n(a)-f_n(b)|~\leq~ L|a-b|+\gamma_n,\qquad\forall a,b\in[0,T],
	\end{equation*}
	for constants $M,L$ and a sequence $\gamma_n\downarrow 0.$ Then
	any subsequence of $\{f_n\}_{n\in\Nat}$ has a subsequence that converges to an $L$-Lipschitz function $f$ uniformly on $[0,T]$ with $|f(0)|\leq M.$
\end{lemma}

The next lemma is used to prove Theorem~\ref{thm:mainDiffusion}. Kurtz~\cite{kurtz:1978kq}
derives diffusion approximations for variety of continuous Markov chains and the following lemma is a special case. We use it to obtain the diffusion limit of $\{\widehat{U}_n(t)\}_{n\in\mathbb N}$ in the proof of Theorem~\ref{thm:mainDiffusion}.
\begin{lemma}\label{lemma:kurtz}
  Consider a sequence of real-valued Markov processes $\{U_N(t)\}_{N\in\Nat}$ which satisfies
  \begin{equation*}
    U_N(t)~=~ u_0 + \frac{1}{N}\,A_N\!\left( N \int_0^t f_{N,1}(U_N(s)) \ud s \right)
    -\frac{1}{N}\,S_N\!\left(N\int_0^t f_{N,-1}(U_N(s))\ud s\right),
  \end{equation*}
  where $A_N(\,\cdot\,)$ and $S_N(\,\cdot\,)$ are independent Poisson processes with rate $1$, and $f_{N,i}$ are positive valued continuous functions for $i=\pm 1$. Suppose that there exist a constant $M>0$ and functions $f_1$ and $f_{-1}$ such that 
  \begin{equation*}
    f_{N,i}(x) \leq M, \quad
    |f_{N,i}(x)-f_i(x)|\leq \frac{M}{N},\quad\textrm{and}\quad
    |\sqrt{f_{i}(x)}-\sqrt{f_{i}(y)}|^2\leq M|x-y|^2  
  \end{equation*}
  for $i=\pm 1$. Let $F(x)=f_1(x)-f_{-1}(x)$ and also assume that 
    $\left|F^\prime(x)\right|\leq M$, and $\left|F^{\prime\prime}(x)\right|\leq M$.
  Then we have
  \begin{equation*}
    \sqrt{N}\left(U_N(t)-u(t)\right)~\Rightarrow~ V(t),
  \end{equation*}
  where $u(t)$ is a function satisfying 
  \begin{equation*}
    u(t)=u_0+\int_0^t f_{1}(u(s))\ud s-\int_0^t f_{-1}(u(s))\ud s
  \end{equation*} and $V(t)$ is a stochastic process given by
  \begin{equation*}
    V(t)~=~\int_0^t \sqrt{f_1(u(s))} \ud B^{(1)}(s)-\int_0^t \sqrt{f_{-1}(u(s))} \ud B^{(2)}(s)
    + \int_0^t F^\prime(u(s)) V(s) \ud s,
  \end{equation*}
  where $B^{(1)}(t)$ and $B^{(2)}(t)$ are independent Wiener processes.
\end{lemma}

\bibliographystyle{acmtrans-ims}	
\bibliography{randomized_LQF}

\end{document}